\documentclass[11pt]{article}
\usepackage{amssymb,amsmath,amsthm}
\usepackage{indentfirst}
\usepackage{hyperref}
\usepackage{cite}
\usepackage{graphicx}
\usepackage{epstopdf}
\usepackage{subfigure}
\usepackage{caption}
\usepackage{enumerate}
\usepackage{bm}

\usepackage{tikz}
\usepackage{pgflibraryarrows}
\usepackage{pgflibrarysnakes}
\usepackage{appendix}
\allowdisplaybreaks[4]

\makeatother
\usepackage{amsfonts}
\usepackage{mathrsfs}

\newtheorem{theorem}{Theorem}[section]

\newtheorem{lemma}[theorem]{Lemma}
\newtheorem{proposition}[theorem]{Proposition}
\newtheorem{definition}[theorem]{Definition}
\newtheorem{remark}[theorem]{Remark}

\title{ \bf{A nonlocal diffusion single population model in advective environment}
      }
\date{}
\author{ Yaobin Tang, 
 Binxiang Dai\footnote{\small Corresponding author.
        Email address: \url{bxdai@csu.edu.cn} (B. Dai).}\\
        {\small School of Mathematics and Statistics,\ HNP-LAMA,}\\
        {\small Central South University,}\\
        {\small Changsha, Hunan 410083, P. R. China }}

\begin{document}
\captionsetup[figure]{labelfont={bf},labelformat={default},labelsep=period,name={Fig.}}
\maketitle

\begin{abstract}
This paper is devoted to  a nonlocal reaction-diffusion-advection model that describes the spatial dynamics of freshwater organisms in a river with a directional motion. Our goal is to investigate how the advection rate affects the dynamic behaviors of species. We first establish the well-posedness of  global solutions, where the regularized problem containing a viscosity term and the re-established maximum principle play an important role. And we then show the existence/nonexistence, uniqueness, and stability of nontrivial stationary solutions by analyzing the principal eigenvalue of integro-differential operator (especially studying the monotonicity of the principal eigenvalue with respect to the advection rate), which enables us to understand the longtime behaviors of solutions and obtain the sharp criteria for persistence or extinction. Furthermore, we study the limiting behaviors of solutions with respect to the advection rate and find that the sufficiently large directional motion will cause species extinction in all situations. Lastly, the numerical simulations verify our theoretical proofs.



\end{abstract}

\section{Introduction\label{part1}}
\setcounter{equation}{0}
The dispersal ways of populations have received much attention recently and are a major focus of most scholars.  The reaction-diffusion equation is a mathematical model commonly used to describe the spatial random diffusion and evolution of biological populations or diseases \cite{Murray1989,Fisher1937,Skellam1951}.

Hutson and Martinez \cite{Huston2003}  found that the model, derived from random diffusion, is too restrictive to model the seed dispersal of a single species with long-range dispersal.  And they thought that the position-jump process is more suitable. Let $u(t,x)$ be the population density at time $t$ and location $x$. Assume that the population $u(t,x)$ can jump from one position $x$ to another position $y$ with a certain probability $J(y-x)$. Then $\int_\mathbb{R} J(x-y)u(t,y)dy$ represents the population density arriving to position $x$ from all other positions and $\int_\mathbb{R} J(y-x)u(t,x)dy$ denotes the population density leaving position $x$ to all other positions. Hence, the nonlocal diffusion,  described by the convolution operator
\[(J*u)(t,x)=d\int_\mathbb{R} J(x-y)u(t,y)dy-du(t,x),\]
 can be introduced to model the dispersion of individuals \cite{Crow1970,Othmer1988,Cortazar2007}, where $d$ is the dispersal rate and the kernel function $J:\mathbb{R}\to\mathbb{R}^+$ is continuous, nonnegative, and has the properties
  \begin{enumerate}[\bf{(J)}:]
 \item \quad\quad\quad\quad$J(0)>0,\int_\mathbb{R} J(x)dx=1$.
\end{enumerate}
Denote $\Omega\subset\mathbb{R}$ is a smooth domain. If the dispersal takes place in $\mathbb{R}$ and species $u$ vanishes outside $\Omega$, then
\begin{equation*}
\begin{cases}
        u_t(t,x)=d\int_\Omega J(x-y)u(t,y)dy-du(t,x),&t>0,x\in{\Omega},\\
         u(t,x)=0,&t>0,x\in\mathbb{R}\backslash{\Omega}
        \end{cases}
 \end{equation*}
is called the nonlocal diffusion equation with homogeneous Dirichlet boundary condition \cite{Cortazar,Andreu-Vaillo}. Biologically, the region $\Omega$ outside  is considered a hostile environment, and any individual that jumps outside of $\Omega$ dies instantly. Note that the boundary condition is not understood in the usual sense, as we are not imposing that $u|_{\partial\Omega}=0$.

If it takes into account the diffusion inside $\Omega$, as we have explained, no individual enters or leaves the domain $\Omega$, then
 \[u_t(t,x)=d\int_\Omega J(x-y)[u(t,y)-u(t,x)]dy, t>0,x\in{\Omega},\]
  is called  the nonlocal diffusion equation with classical Neumann boundary condition in the literature \cite{Cortazar1,Cortazar2}.

There are a lot of researches on nonlocal diffusion problems in biomathematics.  In recent years, the following nonlocal evolution equation
\begin{equation}\label{3a_26}
\begin{cases}
        u_t(t,x)=d\int_\Omega J(x-y)u(t,y)dy-du(t,x)+f(x,u),&t>0,x\in\Omega,\\
         u(0,x)=\phi(x),&x\in\Omega
        \end{cases}
 \end{equation}
has been considered.  When $\Omega=\mathbb{R}^{N}$, several studies focused on the analysis of entire solutions, propagation speeds, and the corresponding traveling wave problems of \eqref{3a_26}, see \cite{Coville2005,Pan2009,Covill2013, Li2010,Li2015,Sun2011}. When $\Omega\subset\mathbb{R}^{N}$ is bounded, there are some works focusing on the model with
Dirichlet/Neumann boundary condition, see \cite{Bates2007,Kao2010,Covill2010,GM2009}. Moreover, the nonlocal problem was applied in various specific situations, such as the studies of asymptotic behaviors \cite{Chasseigne2006,Cortazar2007,Zhang2010,Hetzer2012}, spectral theory \cite{Covill2010,GM2009,Hutson2008}, application in infectious disease models \cite{Ruan2007,Li2014}, and the analysis of systems with Lotka-Volterra type \cite{Lilou2014,Han2020,Wang2021}.



In addition, some researchers are devoted to  the nonlocal problem with a Fisher-KPP type, bistable type or general nonlinear reaction term.  Specifically, Bates \& Zhao \cite{Bates2007} investigated the spatial dynamics of organisms or cells. To be specific, they considered the following model:
\begin{equation*}
\begin{cases}
        u_t(t,x)=d\int_\Omega J(x-y)u(t,y)dy+b(x)u(t,x)+f(x,u),&t>0,x\in\Omega,\\
         u(0,x)=\phi(x),&x\in\Omega,
        \end{cases}
 \end{equation*}
where  $f(x,u)\in C^1(\bar{\Omega}\times \mathbb{R}^+)$ is strictly sublinear with respect to $u$ satisfying $f(x,0)\equiv0$. Subsequently, Sun \cite{Sun2014}  discussed the nonlocal problem \eqref{3a_26} with reaction term $f(x,u)=\lambda g(x)h(u)$, which arises from selection-migration models in genetics. Here $g\in C(\bar{\Omega})$ is sign changed and $h: [0,1] \to\mathbb{R}^+$ is a smooth function such that $h(0)=h(1)=0$ and $h''(u)<0$ for $0 \leq u \leq1$. Moreover, \cite{Berestycki2016} focused on the nonlocal problem \eqref{3a_26} with a KPP-type reaction term $f(x,u)\in C^{1,\alpha}(\mathbb{R}^{2})$ which satisfies
$$\limsup_{|x|\to\infty}\frac{f(x,s)}{s}<0 \ \ \mbox{uniformly in} \ s\geq0.$$
 Nextly, the nonlocal problem \eqref{3a_26} with $f(x,u)=r(x)u(1-\frac{u}{K(x)})$ was studied by Su \cite{Su2019}, where $r(x),K(x)>0$ on $\bar{\Omega}$ denote respectively its intrinsic growth rate and its carrying capacity. And the nonlocal problem \eqref{3a_26} with $f(x,u)=m(x)u-c(x)u^p$ was analyzed by Sun \cite{Sun2021}, where $m(x)$ stands for the intrinsic growth rate and is sign changed,  $c(x)>0$ on $\bar{\Omega}$  measures the capacity of $\Omega$ to support the species.
They provided the existence, uniqueness, and stability of  positive stationary solutions, and obtained the longtime behaviors through certain threshold values. It is worth mentioning that, the selection of threshold value usually depends on the sign of the principal eigenvalue of the operator:
\begin{align*}
(\mathfrak{L}_{\Omega,J}+a)[\varphi(x)]:=d\int_\Omega J(x-y)\varphi(y)dy-d\varphi(x)+a(x)\varphi(x),
\end{align*}
 which is defined on $C^1(\bar{\Omega})$ and $a(x)\in C(\bar{\Omega})\cap L^\infty(\Omega)$.  It is a crucial factor in determining the dynamic behaviors of the solution.

In addition to the spatial diffusion of populations, various environmental factors such as wind direction, river flow, climate, and resource distribution can influence the longtime dynamic behaviors of biological populations and diseases. To account for this phenomenon, researchers include an advection term to describe the directed movement of the population or disease over space.


In 1954, Muller \cite{6} observed that organisms living in streams or rivers are constantly washed downstream or out of the original ecological environment, resulting in a general decline in species, even extinction. But some species can survive in this environment from generation to generation, and this biological phenomenon often is called  "drift paradox". So they  puzzled over a  question: How do species in rivers survive without being washed away? Then, Speirs \& Gurney introduced the reaction-advection-diffusion equation in \cite{7} to describe the dynamic behaviors of single species under the influence of  currents:
\begin{equation*}
\begin{cases}
        u_t(t,x)=du_{xx}(t,x)-qu_x(t,x)+uf(u),&t>0,x\in(0,L),\\
        du_x(t,0)-qu(t,0)=0,u(t,L)=0, &t>0,\\
        u(0,x)=\phi(x),&x\in(0,L).
        \end{cases}
 \end{equation*}
The biological explanation for this is that the population $u$ not only undergoes random diffusion but also experiences advective transport at velocity $q$. They demonstrated that a large enough random diffusion can counterbalance the directional motion caused by  currents, leading to the survival of species. Subsequently,  the persistence of single or multiple populations in advective environments was studied \cite{10,11,12,13,14}.
Some significant and profound results have been obtained for reaction-diffusion problems in advective environments. However, the question remains whether these results can be extended to nonlocal diffusion problems in advective environments. As far as the author knows, in \cite{Coville2006}, the monotonicity and uniqueness of positive solutions of  integro-differential equation
\[\int_\mathbb{R} J(x-y)u(y)dy-u(x)-cu'(x)+f(u)=0, x\in \mathbb{R}\]
was studied. Its positive solution can be regarded as the traveling wave solution $u(t,x)=\phi(x+ct)$ for the following nonlocal problem:
\[u_t(t,x)=\int_\mathbb{R} J(x-y)u(t,y)dy-u(t,x)+f(u),\ (t,x)\in\mathbb{R}^+\times\mathbb{R}.\]
Then, Zhao \& Shen et al. \cite{Zhang2020,DeLeenheer} discussed a nonlocal reaction-diffusion equation in shifting habitat:
\[u_t(t,x)=\int_\mathbb{R} J(x-y)u(t,y)dy-u(t,x)+f(x-ct,u)u, (t,x)\in\mathbb{R}^+\times\mathbb{R}.\]
Biologically, it represents that the population will change with the moving habitat of speed $c$ due to climate change. Considering the function $v(t, \zeta)$ that satisfies the form $u(t,x)=v(t,x-ct)$ with $v(t, x)$ being differentiable and $\zeta= x-ct$,  we find that $v(t, \zeta)$ satisfies
\begin{align}\label{3a_27}
u_t(t,\zeta)=cu_\zeta(t,\zeta)+\int_\mathbb{R} J(\zeta-y)u(t,y)dy-u(t,\zeta)+f(\zeta,u)u, (t,\zeta)\in\mathbb{R}^+\times\mathbb{R}.
\end{align}
And then,  the nonexistence/existence of nontrivial stationary solutions of \eqref{3a_27} was obtained by studying the spectral theory of integro-differential operator:
 \[(\mathfrak{L}_{\Omega,J,q}+a)[\varphi(x)]:=d\int_\Omega J(x-y)\varphi(y)dy-q\varphi'(x)+a(x)\varphi(x),\]
which was studied by \cite{Covill2010,Berestycki2016,DeLeenheer,Shen2015}.

In addition, observe that  \cite{Li2017JW} considered fresh water organisms living in a river domain ${\Omega}=(l_1,l_2)$.  Assume that these organisms not only can move within the river due to random diffusion, but are also subject to directed movement arising from currents. Concretely,
  these organisms are affected by the velocity of currents from left to right
 and they can not survive on land $(-\infty,l_1)$ and in the ocean $(l_2,+\infty)$:
 \begin{equation*}\label{4a_37}
\begin{cases}
        u_t=d[\int_{l_1}^{l_2} J(x-y)u(t,y)dy-u(t,x)]-(c(x)u(t,x))_x,&t>0,x\in(l_1,l_2),\\
         u(t,l_1)=0,&t>0,
        \end{cases}
 \end{equation*}
  where $c(x)>0$ on $[l_1,l_2]$ is the velocity function of water. They emphasized the importance of boundary condition $u(t,l_1)=0$ for $t>0$, which ensures that the initial value problem is well-posed. Their paper mainly discussed the spectrum of the linearized problem, and studied the sign of principal eigenvalue.

In this paper, we still consider fresh water organisms living in a river domain ${\Omega}=(l_1,l_2)$ and
intend to study
the spatial dynamics behaviors of fresh water organisms in advective environment.
If we further assume that these organisms can not  enter or leave the domain $\Omega$ and are affected by the velocity $q$ of currents from left to right, then  we study the following nonlocal Neumann advection-diffusion model:
\begin{equation}
\label{3a_11}
\begin{cases}
        u_t=d\int_\Omega J(x-y)[u(t,y)-u(t,x)]dy-qu_x+f(t,x,u),&t>0,x\in{\Omega},\\
        u(t,l_1)=0,&t>0,\\
       u(0,x)=u_0(x),&x\in{\Omega}.
        \end{cases}
 \end{equation}

If we further assume that these organisms are affected by the velocity $q$ of currents from left to right
 and they can not survive on land $(-\infty,l_1)$ and in the ocean $(l_2,+\infty)$. Then, we study the following nonlocal Dirichlet advection-diffusion model:
 \begin{equation}
\label{3a_1}
\begin{cases}
        u_t=d[\int_\Omega J(x-y)u(t,y)dy-u(t,x)]-qu_x+f(t,x,u),&t>0,x\in\Omega,\\
         u(t,l_1)=0,&t>0,\\
       u(0,x)=u_0(x),&x\in\Omega.
        \end{cases}
 \end{equation}
Here $\Omega$ is a bounded open domain, $q>0$ is the advection rate and the initial function $u_0(x)$ satisfies
\begin{align}\label{3a_2}
u_0(x)\in C^1(\bar{\Omega}),u_0(l_1)=0, u_0(x)\geq,\not\equiv 0 \ \ \mbox{in} \ \ {\Omega}.
\end{align}
\begin{remark}
\begin{enumerate}[(1)]
\item The  boundary condition $u(t,l_1)=0$ for $t>0$ is to ensure the initial value problem
is well-posed. In fact, the equation \eqref{3a_11} or \eqref{3a_1} are essentially a perturbation of a transport equation by an integral operator. When $d\equiv0$, we find that the equation is a transport equation in a bounded domain, which need a boundary condition to ensure the well-posedness of solution.

 \item If these organisms are affected by the velocity $q$ of currents from right to left, the boundary condition will be rewritten as $u(t,l_2)=0$ for $t>0$. Its discussion is similar to the analysis of \eqref{3a_11} or \eqref{3a_1}.
\end{enumerate}
\end{remark}

Assume further that the kernel function $J:\mathbb{R}\to\mathbb{R}^+$  satisfies the properties
 \begin{enumerate}[\bf{(J0)}:]
  \item \quad\quad\quad There exist $\mu,M >0$ such that $J(x)<e^{-\mu|x|}$ and $|J'(x)|<e^{-\mu|x|}$ for $|x|>M$.
\end{enumerate}
  \begin{enumerate}[\bf{(J1)}:]
 \item \quad\quad\quad\quad\quad\quad\quad\quad \quad\quad\quad\quad\quad\quad   $J$ is symmetric, $\sup J(x)<\infty$.
\end{enumerate}
Meanwhile, assume that, for any given $\tilde{B}>\tilde{b}$, there exists constant $N>0$ such that for any $(t,x),(s,y)\in \mathbb{R}^+\times\overline{\Omega}$ and $u\in[\tilde{b},\tilde{B}]$, the function $f(t,x,u)$ satisfies
\begin{enumerate}[\bf{(F1)}:]
 \item \quad\quad\quad\quad $f,f_t,f_x,f_u$ are continuous in $t,x,u$; $f,f_u$ are Lipschitz continuous in $u$;
\end{enumerate}
\begin{enumerate}[\bf{(F2)}:]
 \item  \quad\quad\quad\quad$f(t,x,0)=0$;  $f(t,x,u)<0,\forall u>N$.
\end{enumerate}

The purpose of this paper is to show the existence and uniqueness of  global solutions, the existence/nonexistence,  uniqueness and stability of nontrivial stationary solutions, the longtime behaviors of solutions, and get the sharp criteria for persistence or extinction. In addition,  the effects of advection rate on the dynamic behaviors of solutions will be discussed.

We point out that, we adopt a completely new method to get the well-posedness of global solution of \eqref{3a_11} and \eqref{3a_1}. Due to the lack of the compactness and regularity of the nonlocal dispersal operator, the well-posed global solution of \eqref{3a_11} and \eqref{3a_1} will be proved by approximating the solution of a regularized problem containing a viscosity term, where the upper and lower solution methods, maximum principle  and iterative techniques play an important role.
Moreover, in the process of discussing nontrivial stationary solutions, it is necessary to consider the spectral theory of integro-differential operator. In this paper, we aim to solve an open problem posed in \cite{DeLeenheer} under certain conditions. Specifically, we prove that if $\Omega$ is bounded and $a(x)\in C(\bar{\Omega})\cap L^\infty(\Omega)$ is nonnegative, then the principal eigenvalue of operator $\mathfrak{L}_{\Omega,J,q}+a$ defined in \eqref{3a_13} is increasing in $|q|>0$. The above conclusion can directly deduce the sharp criteria for persistence or extinction. Meanwhile, the proof methods about the uniqueness of nontrivial stationary solutions of  \eqref{3a_11} or  \eqref{3a_1} are different from the usual methods, since all nontrivial stationary solutions of \eqref{3a_11} or \eqref{3a_1} are $0$ at left boundary point $l_1$.




The remaining parts of this paper are organized as follows. Section \ref{3a_part2} is  devoted to studying the existence and uniqueness of solutions for  \eqref{3a_11} and  \eqref{3a_1}. In Section \ref{3a_part3},  the existence/nonexistence,  uniqueness and stability of nontrivial stationary solutions are obtained.  In Section \ref{3a_part4}, the criteria for persistence or extinction are formulated. And in Section \ref{3a_part5},  the limiting profiles of solution with respect to advection rate $q$ are considered. The last section gives a brief discussion.

\section{The well-posed global solution \label{3a_part2}}
\setcounter{equation}{0}
For convenience, we set
\[Q_T=(0,T)\times\Omega,Q_\infty=\mathbb{R}^+\times\bar{\Omega},\overline{Q_T}=[0,T]\times\bar{\Omega}.\]
In order to obtain the existence of solutions for nonlocal  Neumann problem \eqref{3a_11},  we must consider the following linear  problem:
\begin{equation}
\label{3a_3}
\begin{cases}
        u_t=d\int_\Omega J(x-y)[u(t,y)-u(t,x)]dy-qu_x+c(t,x)u+g(t,x),&t>0,x\in{\Omega},\\
        u(t,l_1)=0,&t>0,\\
       u(0,x)=u_0(x),&x\in{\Omega}.
        \end{cases}
 \end{equation}
where $c(t,x),g(t,x)\in C^1(Q_\infty)\cap L^\infty(Q_\infty)$.

\begin{definition}
We say that $u(t,x)\in W^{1,2}(\overline{Q_T})$ is a local weak solution of \eqref{3a_3} if
\[\int_\Omega\int_0^T \phi u_tdtdx
 =\int_\Omega\int_0^T [d\int_\Omega J(x-y)[u(t,y)-u]dy -q u_x+c u+g]\phi dtdx\]
for all $\phi\in C_0^\infty(\overline{Q_T})$, where $$ C_0^\infty(\overline{Q_T}):=\{\phi\in C^\infty(\overline{Q_T})|  \ D^k_x\phi(t,l_2)=0 \ \forall t\in(0,T), \forall k\in\mathbb{N}\}.$$
\end{definition}

Due to the lack of compactness and regularity of nonlocal operators, many classical methods cannot be used directly to nonlocal dispersal problems. To overcome these difficulties, we shall consider the following problem regularized by the usual viscosity term:
\begin{equation}
\label{3a_4}
\begin{cases}
        u_t^\epsilon-\epsilon u^\epsilon_{xx}=d\int_\Omega J(x-y)[u^\epsilon(t,y)-u^\epsilon(t,x)]dy-qu^\epsilon_x+c(t,x)u^\epsilon+g(t,x),&t>0,x\in{\Omega},\\
        u^\epsilon(t,l_1)=\frac{\partial u^\epsilon}{\partial\bf{\overrightarrow{n}}}(t,l_2)=0,&t>0,\\
         u^\epsilon(0,x)=u_0(x),&x\in\Omega,
        \end{cases}
 \end{equation}
where $\overrightarrow{n}$ denotes the outward pointing unit normal along $\partial\Omega$. Our first goal is to show that problem \eqref{3a_4} has a unique local nonnegative solution.

\begin{lemma}\label{3lemma1}
For fixed $\epsilon>0$, if $u_0(x)$ satisfies \eqref{3a_2}, then there exists $T>0$ such that the regularized problem  \eqref{3a_4} admits a unique nonnegative local classical solution $u^\epsilon(t,x)$ satisfying
\begin{align}
0\leq \mbox{ess} \inf_{\bar{\Omega}} &u_0(x)\leq u^\epsilon(t,x)\leq \mbox{ess} \sup_{\bar{\Omega} } u_0(x),\ \  \forall (t,x)\in \overline{Q_T},\label{3a_5}
\end{align}
and
\begin{align}
|u^\epsilon_x(t,x)|\leq R,\ \  \forall (t,x)\in \overline{Q_T},\label{3a_8}
\end{align}
where
\[R:=\max\{|\min_{\bar{\Omega}}u_0'(x)|, |\max_{\bar{\Omega}}u_0'(x)|,\max_{[0,T]}u^\epsilon_x(t,l_1)\}.\]
\end{lemma}
\begin{proof}
 Set $X=L^\infty(0,T;H^1(\Omega))$. For each $v^\epsilon\in X$, consider the linear problem
\begin{equation}\label{3a_23}
\begin{cases}
        u_t^\epsilon-\epsilon u_{xx}^\epsilon=d\int_\Omega J(x-y)[v^\epsilon(t,y)-v^\epsilon(t,x)]dy-qv^\epsilon_x+c(t,x)v^\epsilon+g(t,x),&t\in(0,T),x\in\Omega,\\
           u^\epsilon(t,l_1)=\frac{\partial u^\epsilon}{\partial\bf{\overrightarrow{n}}}(t,l_2)=0,&t\in(0,T),\\
         u^\epsilon(0,x)=u_0(x),&x\in\Omega.
        \end{cases}
 \end{equation}
 Notice that the right-hand side of the first equation for \eqref{3a_23} is bounded on $L^2$ since nonlocal operator is bounded on $X$ and $v^\epsilon_x\in L^\infty(0,T;L^2(\Omega))$. Then the standard $L^p$ theory implies that \eqref{3a_23} admits a unique solution $u^\epsilon\in L^2(0,T;H^2(\Omega))
 $.
Similarly, choosing $\tilde{v}^\epsilon\in X$, and then $\tilde{u}^\epsilon\in L^2(0,T;H^2(\Omega))$ solves
\begin{equation*}
\begin{cases}
        \tilde{u}_t^\epsilon-\epsilon \tilde{u}_{xx}^\epsilon=d\int_\Omega J(x-y)[\tilde{v}^\epsilon(t,y)-\tilde{v}^\epsilon(t,x)]dy-q\tilde{v}^\epsilon_x+c(t,x)\tilde{v}^\epsilon+g(t,x),&t\in(0,T),x\in\Omega,\\
           u^\epsilon(t,l_1)=\frac{\partial u^\epsilon}{\partial\bf{\overrightarrow{n}}}(t,l_2)=0,&t\in(0,T),\\
       \tilde{u}^\epsilon(0,x)=u_0(x),&x\in\Omega.
        \end{cases}
 \end{equation*}
Hence, for given $z^\epsilon:=v^\epsilon-\tilde{v}^\epsilon$, we find that $w^\epsilon:=u^\epsilon-\tilde{u}^\epsilon$ satisfies
\begin{equation*}
\begin{cases}
        w_t^\epsilon-\epsilon w_{xx}^\epsilon=d\int_\Omega J(x-y)[z^\epsilon(t,y)-z^\epsilon(t,x)]dy-qz^\epsilon_x+c(t,x)z^\epsilon,&t\in(0,T),x\in\Omega,\\
           w^\epsilon(t,l_1)=\frac{\partial w^\epsilon}{\partial\bf{\overrightarrow{n}}}(t,l_2)=0,&t\in(0,T),\\
      w^\epsilon(0,x)=0,&x\in\Omega.
        \end{cases}
 \end{equation*}
 Since it follows from  \cite[Theorem 2, Chapter 5.9.2]{Evans} that
 \[\mbox{ess}\sup_{[0,T]}\| w^\epsilon\|_{H^1(\Omega)}\leq C(\epsilon)
\|w^\epsilon\|_{W^{1,2}(0,T;L^2(\Omega))},\]
we have the  following estimation
 \begin{align*}
\mbox{ess}\sup_{[0,T]}&\| w^\epsilon\|_{H^1(\Omega)}\leq C(\epsilon)
\|w^\epsilon\|_{W^{1,2}(0,T;L^2(\Omega))}\leq C(\epsilon)\|w^\epsilon\|_{W^{2,2}(0,T;L^2(\Omega))}\\
&\leq C(\epsilon)(d\|\int_\Omega  J(x-y)[z^\epsilon(t,y)-z^\epsilon(t,x)]dy\|_{L^2(0,T;L^2(\Omega))}+\| qz^\epsilon_x+c(t,x)z^\epsilon\|_{L^2(0,T;L^2(\Omega))})\\
&\leq C(\epsilon)(2d\|J\|_{L^2(0,T;L^2(\Omega))}\|z^\epsilon\|_{L^2(0,T;L^2(\Omega))}+q\|z^\epsilon\|_{L^2(0,T;H^1(\Omega))}
+\|c\|_{C^1(\overline{Q_T})}\|z^\epsilon\|_{L^2(0,T;L^2(\Omega))})\\
&\leq C(\epsilon)\| z^\epsilon\|_{L^2(0,T;H^1(\Omega))}\leq C(\epsilon) T^{1/2} \mbox{ess} \sup_{[0,T]}\| z^\epsilon\|_{H^1(\Omega)},
\end{align*}
where we use embedding theory and standard $L^p$ theory.

As long as $T\in(0,\frac{1}{4 C^2(\epsilon)}]$, we have $\|u^\epsilon-\tilde{u}^\epsilon\|_{L^\infty(0,T;H^1(\Omega))}\leq \frac{1}{2} \|v^\epsilon-\tilde{v}^\epsilon\|_{L^\infty(0,T;H^1(\Omega))}$, which means that the mapping $A:X\to X$ for $A(v^\epsilon)=u^\epsilon$ is a contraction mapping. Based on Banach's Fixed Point Theorem, mapping $A$ has a unique fixed point, and so  \eqref{3a_4} has a unique local-in-time solution $u^\epsilon\in L^\infty(0,T;H^1(\Omega))$.  By embedding theory and standard Schauder theory, we can obtain that $u^\epsilon$ also is a local classical solution.

We next examine the inequalities \eqref{3a_5}. Let $M=\mbox{ess} \sup_{\bar{\Omega} } u_0(x)>0$. The function $w(t,x):=u^\epsilon(t,x)-M$ satisfies
\begin{equation*}
\begin{cases}
        w_t=\epsilon w_{xx}+d\int_\Omega  J(x-y)[w(t,y)-w(t,x)]dy-qw_x+c(t,x)(w+M)+g(t,x),&t\in(0,T),x\in\Omega,\\
          w(t,l_1)=-M<0,\frac{\partial w}{\partial\bf{\overrightarrow{n}}}(t,l_2)=0,&t\in(0,T),\\
      w(0,x)=u_0(x)-M\leq0,&x\in\Omega.
        \end{cases}
 \end{equation*}
{\bf Claim:} The inequality $w(0,x)\leq0$ implies $w(t,x)\leq0$ for all $(t,x)\in \overline{Q_T}$. Define function $w^1=e^{-Kt}w$ and suppose that there is $(t_*,x_*)\in \overline{Q_T}$ such that $w^1(t_*,x_*)=\sup_{\overline{Q_T}} w^1(t,x)>0$. Obviously, $t_*\in(0,T]$ and $x_*\in(l_1,l_2]$, and then $w^1_t(t_*,x_*)\geq0$. Meanwhile, we can choose large enough constant $K>0$ satisfying $$(c(t_*,x_*)-K)w^1(t_*,x_*)+(c(t_*,x_*)M+g(t_*,x_*))e^{-Kt_*}<0 \ \mbox{on} \  \overline{Q_T},$$
 since $c(t,x),g(t,x)\in C^1(\overline{Q_T})$.

  If $x_*\in\Omega$, then it follows from  $ w^1_{xx}(t_*,x_*)\leq0,w^1_x(t_*,x_*)=0$ and $w^1(t_*,x)\leq w^1(t_*,x_*)$ on $\Omega$ that
\begin{align*}
0\leq w^1_t(t_*,x_*)=&e^{-Kt_*}w_t(t_*,x_*)-Ke^{-Kt_*}w(t_*,x_*)\\
&=\epsilon w^1_{xx}(t_*,x_*)+d\int_\Omega  J(x^*-y)[w^1(t_*,y)-w^1(t_*,x_*)]dy-qw^1_x(t_*,x_*)\\
&\quad+(c(t_*,x_*)-K)w^1(t_*,x_*)+(c(t_*,x_*)M+g(t_*,x_*))e^{-Kt_*}\\
&\leq(c(t_*,x_*)-K)w^1(t_*,x_*)+(c(t_*,x_*)M+g(t_*,x_*))e^{-Kt_*}<0,
\end{align*}
which is impossible. If $x_*=l_2$, the boundary condition $\frac{\partial  w}{\partial\bf{\overrightarrow{n}}}|_{x=l_2}=0$ implies $w^1_x(t_*,x_*)=0$. We slightly modify the proof by setting $w^1(t,x_n)=w^1(t,x_*)-\dfrac{1}{n}$. By
\cite[Lemma 3.1]{Pudelko}, the sequence $\{w^1(t,x_n)\}$ has the following properties:
\[\lim\limits_{n\to\infty}w^1_x(t,x_n)=0; \ \limsup\limits_{n\to\infty}w^1_{xx}(t,x_n)\leq0; \ \limsup\limits_{n\to\infty}\int_\Omega  J(x_n-y)[w^1(t,y)-w^1(t,x_n)]dy\leq0.\]
Thus, letting $t=t_*$, as $n\to \infty$, we have
 \begin{align*}
0\leq w^1_t(t_*,x_*)\leq(c(t_*,x_*)-K)w^1(t_*,x_*)+(c(t_*,x_*)M+g(t_*,x_*))e^{-Kt_*}<0,
\end{align*}
 which is impossible. And the claim is true. Thus $u^\epsilon(t,x)\leq M$ on $\overline{Q_T}$, and the proof of the  inequality $0\leq \mbox{ess} \inf_{\bar{\Omega}} u_0(x)\leq u^\epsilon(t,x)$ is analogous.

Finally, we verify the inequality \eqref{3a_8}. Denote $z(t,x)=u^\epsilon_x(t,x)$. By assumption ${\bf{(J1)}}$, we know that $z(t,x)$ satisfies
\begin{equation*}
\begin{cases}
       z_t=\epsilon z_{xx}+d\int_\Omega  J(x-y)[z(t,y)-z(t,x)]dy-qz_x+cz+c_xu^\epsilon+g_x-m(t,x),&t\in(0,T),x\in\Omega,\\
       z(t,l_1)=u^\epsilon_x(t,l_1)\geq0, z(t,l_2)=0,&t\in(0,T),\\
      z(0,x)=(u_0)'(x),&x\in\Omega,
        \end{cases}
 \end{equation*}
 where
\begin{align}\label{3a_29}
m(t,x):=du^\epsilon(t,l_2)J(x-l_2)-du^\epsilon(t,l_1)J(x-l_1)+d\int_\Omega  J_x(x-y)dyu^\epsilon
\end{align}
 is uniformly bounded.
Since $u^\epsilon$ is a uniformly bounded classical solution, $c_xu^\epsilon+g_x-m(t,x)$ is a regular function, and then
 \[|c_xu^\epsilon+g_x-m(t,x)|\leq \|c\|_{C^1}\|u_0\|_\infty+\|g\|_{C^1}+\|m\|_\infty:=R_1.\]
Let $v(t,x):=(R-z(t,x))e^{-Kt}$ for some large enough $K>0$. Clearly, $v(0,x)=R-z(0,x)\geq0$ and $v(t,x)\geq0$ on $\partial\Omega$. We shall show $v(t,x)\geq0$ in $\overline{Q_T}$. Conversely, if there is $(t_*,x_*)\in \overline{Q_{T}}$ such that $v(t_*,x_*)=\inf_{\overline{Q_{T}}} v(t,x)<0$. The choice of $R$ deduces $t_*\in(0,T]$ and $v_t(t_*,x_*)\leq0$. The boundary condition means $x_*\in(l_1,l_2)$, $v_x(t_*,x_*)=0$ and $v_{xx}(t_*,x_*)\geq0$. In view of the assumption ${\bf{(J1)}}$, we know  that
\[d\int_\Omega  J(x_*-y)[v(t_*,y)-v(t_*,x_*)]dy\geq0.\]
Then $v(t_*,x_*)$ satisfies
\begin{align*}
0\geq v_t(t_*,x_*)=&-Kv(t_*,x_*)+\epsilon v_{xx}(t_*,x_*)+d\int_\Omega  J(x_*-y)[v(t_*,y)-v(t_*,x_*)]dy-qv_x(t_*,x_*)\\
&+c(t_*,x_*)v(t_*,x_*)+[c_x(t_*,x_*)u^\epsilon(t_*,x_*)+g_x(t_*,x_*)-m(t_*,x_*)-Rc(t_*,x_*)]e^{-Kt_*}\\
\geq&(c(t_*,x_*)-K)v(t_*,x_*)-(R_1+R\|c\|_{C^1})e^{-Kt_*}\\
\geq &(\|c\|_{C^1}-K)v(t_*,x_*)-(R_1+R\|c\|_{C^1})e^{-Kt_*}>0,
\end{align*}
provided that $K>0$ is a large enough constant such that
\[(\|c\|_{C^1}-K)v(t_*,x_*)-(R_1+R\|c\|_{C^1})e^{-Kt_*}>0.\]
This is a contradiction, and so, $z(t,x)\leq R$ on $\overline{Q_T}$. Similarly, we also can prove $z(t,x)\geq -R$ on $\overline{Q_T}$. The proof is completed.
\end{proof}

In the following, we prove that the solution $u^\epsilon(t,x)$ of regularized problem \eqref{3a_4} converges to a weak solution of problem \eqref{3a_3} as $\epsilon\to0$. We will apply  Aubin-Lions-Simon compactness theorem (Proposition \ref{3proposition4}) and Lebesgue dominated theorem to complete this proof.

\begin{proposition}[\cite{Simon}, Theorem 5]\label{3proposition4}
For $1\leq p_1,p_2\leq \infty$, and $T>0$, assume that Banach spaces $X,B,Y$ satisfy $X\hookrightarrow\hookrightarrow B\hookrightarrow Y$. If $A$ is bounded in $W^{1,p_1}(0,T; Y)$ and $L^{p_2}(0,T; X)$, then $A$ is relatively compact in $L^{p_2}(0,T; B)$ if $1\leq p_2<\infty$ and in $C(0,T; B)$ if $p_2=\infty$, where $X\hookrightarrow\hookrightarrow B$ means that $X$ is compactly embedded in $B$.
\end{proposition}

\begin{theorem}
Assume that $\bf{(J)(J1)}$ holds. Then \eqref{3a_3} admits a local weak solution $u(t,x)\in L^\infty(\overline{Q_T})$, where $T$ is given by Lemma \ref{3lemma1}.
\end{theorem}
\begin{proof}
First of all, we use Proposition \ref{3proposition4}  to show that $\{u^\epsilon(t,x)|\epsilon\in(0,1]\}$ is relatively compact in $C(0,T; L^2(\Omega))$. Choosing $p_1=p_2=\infty$ and $X=H^1(\Omega),B=L^2(\Omega),Y=H^{-1}(\Omega)=(H^1_0(\Omega))^*$, where we define $\|u\|_{H^{-1}(\Omega)}=\sup\{\int_\Omega(DuD\varphi+u\varphi)dx|\varphi\in C^\infty_0(\Omega)\}$. Obviously, $H^1(\Omega)\hookrightarrow\hookrightarrow L^2(\Omega)\hookrightarrow H^{-1}(\Omega)$. Then for any $\varphi\in C^\infty_0(\Omega)$,
\begin{align*}
|\int_\Omega(Du^\epsilon D\varphi+u^\epsilon\varphi)dx|&\leq |\int_\Omega Du^\epsilon D\varphi dx|+|\int_\Omega u^\epsilon\varphi dx|\leq|\int_\Omega u^\epsilon D^2\varphi dx|+|\int_\Omega u^\epsilon\varphi dx|\\
&\leq \|u_0\|_\infty(\int_\Omega |D^2\varphi|+|\varphi| dx)\leq 2\|u_0\|_\infty \|\varphi\|_{C_0^\infty}|\Omega|<\infty.
\end{align*}
Notice that
\begin{align*}
\int_\Omega |u^\epsilon_tD^2\varphi| dx&\leq \int_\Omega\left| [\epsilon u_{xx}^\epsilon+d\int_\Omega J(x-y)[u^\epsilon(t,y)-u^\epsilon(t,x)]dy-qu^\epsilon_x+c(t,x)u^\epsilon+g(t,x)]D^2\varphi\right| dx\\
&\leq  \int_\Omega \left[\epsilon|u^\epsilon D^4\varphi|+ 2d\|u_0\|_\infty \|\varphi\|_{C_0^\infty} +|qu^\epsilon D^3\varphi|+\|c\|_\infty |u^\epsilon D^2\varphi|\right]dx+ \int_\Omega \|g\|_\infty |D^2\varphi|dx\\
&\leq  \|\varphi\|_{C_0^\infty}|\Omega|\|u_0\|_\infty(1+2d+|q| +\|c\|_\infty)+ \|\varphi\|_{C_0^\infty}|\Omega|\|g\|_\infty<\infty.
\end{align*}
Similarly,
$\int_\Omega |u^\epsilon_t\varphi| dx\leq  \|\varphi\|_{C_0^\infty}|\Omega|\|u_0\|_\infty(1+2d+|q| +\|c\|_\infty)+ \|\varphi\|_{C_0^\infty}|\Omega|\|g\|_\infty<\infty$. Thus, we have
\begin{align*}
|\int_\Omega (Du^\epsilon_tD\varphi+u^\epsilon_t\varphi) dx|\leq  \int_\Omega |u^\epsilon_tD^2\varphi|dx+\int_\Omega|u^\epsilon_t\varphi| dx <\infty.
\end{align*}
The above estimates have applied inequalities \eqref{3a_5} and integrating by part formula. It means that $\|u^\epsilon,u^\epsilon_t\|_{H^{-1}(\Omega)}<\infty$, namely, there exists a constant $C>0$ independent of $\epsilon$ such that
\[\|u^\epsilon\|_{W^{1,\infty}(0,T; H^{-1}(\Omega))}=\mbox{ess} \  sup_{[0,T]}\|u^\epsilon\|_{H^{-1}(\Omega)}+\mbox{ess} \   sup_{[0,T]}\|u^\epsilon_t\|_{H^{-1}(\Omega)}\leq C<\infty.\]
And so, $\{u^\epsilon(t,x)|\epsilon\in(0,1]\}$ is  bounded in $W^{1,\infty}(0,T; H^{-1}(\Omega))$.

Obviously, $|\int_\Omega (u^\epsilon)^2 dx|\leq \|u_0^\epsilon\|_\infty^2|\Omega|$ and $|\int_\Omega (u^\epsilon)_x^2 dx|\leq R^2|\Omega|$ by Lemma \ref{3lemma1}, that is to say, $\{u^\epsilon(t,x)|\epsilon\in(0,1]\}$ is bounded in $L^{\infty}(0,T; H^{1}(\Omega))$. Hence, Proposition \ref{3proposition4} ensures that $\{u^\epsilon(t,x)|\epsilon\in(0,1]\}$ is relatively compact in $C(0,T; L^2(\Omega))$, and then there exist $\{\epsilon_n\}_{n\in\mathbb{N}}\subset(0,1]$ and function $u$ such that $u^{\epsilon_n}$ converges to $u$ in  $C(0,T; L^2(\Omega))$ as $\epsilon_n\to 0$. Combining the inequalities \eqref{3a_5}, we have $u\in L^\infty(\overline{Q_T})$.

Now, we show that  $u(t,x)\in L^\infty(\overline{Q_T})$ is a local weak  solution of \eqref{3a_3}. Multiplying the first equation of \eqref{3a_4} by $\phi\in C_0^\infty(\overline{Q_T})$, where $$ C_0^\infty(\overline{Q_T}):=\{\phi\in C^\infty(\overline{Q_T})|  D^k_x\phi(t,l_2)=0 \ \forall t\in(0,T), \forall k\in\mathbb{N}\},$$ and integrating over $[0,T]\times\Omega$, then using the integrating by part formula and the  condition $\bf{(J1)}$, we deduce
\begin{align*}
\int_\Omega &\phi(T,x)u^{\epsilon_n}(T,x)dx-\int_\Omega \phi(0,x)u^{\epsilon_n}_0(x)dx-\int_\Omega\int_0^T \phi_tu^{\epsilon_n}dtdx\\
 =&d\int_\Omega\int_0^T \int_\Omega J(x-y)[\phi(t,y)-\phi(t,x)]dyu^{\epsilon_n}dtdx+q\int_\Omega\int_0^T \phi_xu^{\epsilon_n}dtdx\\&+\epsilon_n\int_\Omega\int_0^T \phi_{xx}u^{\epsilon_n}dtdx+\int_\Omega\int_0^T (c\phi u^{\epsilon_n}+g\phi) dtdx+\epsilon_n\int_0^Tu^{\epsilon_n}_x(t,l_1)\phi(t,l_1)dt.
\end{align*}
As $\epsilon_n\to 0$, applying the Lebesgue dominated theorem, we obtain
\begin{align*}
\int_\Omega &\phi(T,x)u(T,x)dx-\int_\Omega \phi(0,x)u_0(x)dx-\int_\Omega\int_0^T \phi_tudtdx\\
 &=d\int_\Omega\int_0^T \int_\Omega J(x-y)[\phi(t,y)-\phi(t,x)]dyu(t,x)dtdx+\int_\Omega\int_0^T (q\phi_xu+c\phi u+g\phi) dtdx.
\end{align*}
 Again using the integrating by part formula, $u$ satisfies
\begin{align*}
\int_\Omega\int_0^T \phi u_tdtdx
 =\int_\Omega\int_0^T [d\int_\Omega J(x-y)[u(t,y)-u]dy -q u_x+c u+g]\phi dtdx-\int_0^Tu(t,l_1)\phi(t,l_1)dt,
\end{align*}
 which verifies that $u$ is a weak solution of \eqref{3a_3} since the boundary condition $u(t,l_1)=0$. The proof is completed.
\end{proof}

\begin{remark}\label{3remark2}
As shown in the proof of \cite[Lemma 2,2]{tang2023freeboundry}, since $f(t,x,u)$ satisfies \textbf{(F1)(F2)}, we know that $u\in C^1(\overline{Q_T})$ and $u$ is also a classical solution of \eqref{3a_3}.
\end{remark}

\begin{lemma}[Maximum principle]\label{3lemma211}
For any given $u_0(x)\geq0$ and $T>0$, assume that, for some $c(t,x)\in L^\infty(\overline{Q_T})$, function $u(t,x)\in C^1(\overline{Q_T})$ satisfies
\begin{equation}\label{3a_6}
\begin{cases}
        u_t\geq d\int_\Omega J(x-y)[u(t,y)-u(t,x)]dy-qu_x+c(t,x)u,&t\in(0,T),x\in{\Omega},\\
        u(t,l_1)\geq0,&t\in(0,T),\\
       u(0,x)=u_0(x)\geq0,&x\in{\Omega}.
        \end{cases}
 \end{equation}
Then $u(t,x)\geq0$ in $\overline{Q_T}$. Moreover, $u(t,x)>0$ in $(0,T]\times(l_1,l_2]$ if $u_0(x)\not\equiv0$ in $\bar{\Omega}$.
\end{lemma}
\begin{proof}
Choose a negative constant $k$ such that $k+c(t,x)<0$ in $\overline{Q_T}$. Then  $w(t,x):=e^{kt}u(t,x)$ satisfies
\begin{align}\label{3a_7}
w_t\geq d\int_\Omega J(x-y)[w(t,y)-w(t,x)]dy-qw_x+(c(t,x)+k)w.
\end{align}
 Suppose that there exists $(t_0,x_0)\in \overline{Q_T}$ such that $w_{\inf}:=w(t_0,x_0)=\inf_{\overline{Q_T}} w(t,x)<0$. Similar to the proof of Lemma \ref{3lemma1}, we deduce
\begin{align*}
0\geq w_t(t_0,x_0)&\geq d\int_\Omega J(x_0-y)[w(t_0,y)-w(t_0,x_0)]dy-q w_x(t_0,x_0)+ (c(t_0,x_0)+k)w(t_0,x_0)\\
&\geq (c(t_0,x_0)+k)w(t_0,x_0)>0,
\end{align*}
which is a contradiction. Thus, $w(t,x)\geq0$ in $\overline{Q_T}$, namely, $u(t,x)\geq0$ in $\overline{Q_T}$.


Now, we are still in a position to prove that $u(t,x)>0$ in $(0,T]\times(l_1,l_2]$ if $u_0(x)\not\equiv0$ in ${\Omega}$. It is equivalent to proving $w(t,x)>0$ in $(0,T]\times(l_1,l_2]$. Suppose that there is $(t^*,x^*)\in(0,T]\times(l_1,l_2]$ such that $w(t^*,x^*)=0$.
We observe that $w(t^*,x)\equiv0$ for any $x\in\bar{\Omega}$. Otherwise, we can find a sequence $\{x_n\}_{n=1}^\infty\subset{\Omega}$ with $x_n\to x^*$ such that $w(t^*,x_n)>0$ and $w(t^*,x_n)\to w(t^*,x^*)$ as $n\to\infty$. Noting that $w_t(t^*,x^*)\leq0$ and $w_x(t^*,x^*)\leq0$, we obtain
\[0\geq w_t(t^*,x^*) \geq d\int_\Omega J(x^*-y)[w(t^*,y)-w(t^*,x^*)]dy-qw_x(t^*,x^*)+(c(t^*,x^*)+k)w(t^*,x^*),\]
which implies $\int_\Omega J(x^*-y)w(t^*,y)dy=0$. In view of $J(0)>0$ and $w(t^*,y)\geq0$ in $\bar{\Omega}$, it is easy to get $w(t^*,y)=0$ for all $y$ near $x^*$, which is a contradiction with $w(t^*,x_n)>0$.
Define $z(s)=w(t+s,x+qs)$. By a simple calculation, we have
\begin{align*}
z'(s)&=w_t(t+s,x+qs)+qw_x(t+s,x+qs)\\
&\geq d\int_\Omega J(x+qs-y)[w(t+s,y)-w(t+s,x+qs)]dy+[c(t+s,x+qs)+k]w(t+s,x+qs)\\
&\geq \inf_{{Q_\infty}}[k+c(t,x)-d]z(s)\triangleq \theta z(s).
\end{align*}
And so, $[e^{-\theta s}z(s)]'\geq 0$, namely, for any $t\in[0,t^*]$, it follows that
\[e^{-\theta (t^*-t)}z(t^*-t)=e^{-\theta (t^*-t)}w(t^*,x+q(t^*-t))\equiv0\geq z(0)=w(t,x)\geq0.\]
Thus, $w(t,x)\equiv0$ in $[0,t^*]\times\bar{\Omega}$, and it contradicts with $u_0(x)\not\equiv0$ in $\bar{\Omega}$. The proof is completed.
\end{proof}

\begin{theorem}\label{3theorem1}
Assume that $\bf{(J)(J1)}$ holds. Then \eqref{3a_3} admits a unique local classical solution $u(t,x)\in C^1(\overline{Q_T})$, where $T$ is given by Lemma \ref{3lemma1}.
\end{theorem}
\begin{proof}
By Remark \ref{3remark2}, we know that \eqref{3a_3} admits a local classical solution $u(t,x)$. Here, we only prove the uniqueness of classical solution $u$. Suppose that for any given $u_0(x)$, \eqref{3a_3} has two different solution $u_1$ and $u_2$. Then $w:=u_1-u_2$ satisfies
\begin{equation*}
\begin{cases}
        w_t=d\int_\Omega J(x-y)[w(t,y)-w(t,x)]dy-qw_x+c(t,x)w,&t\in(0,T),x\in{\Omega},\\
        w(t,l_1)=0,&t\in(0,T),\\
       w(0,x)=0,&x\in{\Omega}.
        \end{cases}
 \end{equation*}
By Lemma \ref{3lemma211}, we have $w\geq 0$ on $\overline{Q_T}$, and $w\leq 0$ on $\overline{Q_T}$.
 And so $u_1-u_2=0$. It is a contradiction with our assumption. The uniqueness is completed.
\end{proof}

In the followings, we shall apply the upper and
lower solution methods to get the existence of solutions of \eqref{3a_11}, so we define the  upper and lower solutions as follows:
\begin{definition}
The function $\bar{u}(t,x)\in C^1(\overline{Q_T})$ is  called the upper solution of \eqref{3a_11} if
\begin{equation*}
\begin{cases}
        \bar{u}_t\geq d\int_\Omega J(x-y)[\bar{u}(t,y)-\bar{u}(t,x)]dy-q\bar{u}_x+f(t,x,\bar{u}),&t>0,x\in{\Omega},\\
         \bar{u}(t,l_1)\geq0,&t\in(0,T),\\
       \bar{u}(0,x)\geq u_0(x),&x\in{\Omega}.
        \end{cases}
 \end{equation*}
Moreover, we can also define a lower solution $\underline{u}(t,x)$  by reversing all the inequalities in the above equation.
\end{definition}

Combining with Lemma \ref{3lemma211},  naturally, it gives a ordering of upper solution and lower solution:
\begin{lemma}
(Comparison principle ) If $\bar{u}(t,x),\underline{u}(t,x)$ are respectively an upper solution and a lower solution of \eqref{3a_11} and $f(t,x,u)$ satisfies ${\bf{(F1)(F2)}}$ with $\tilde{b}= \min_{\overline{Q_T}}\min\{\bar{u},\underline{u}\}$ and $\tilde{B}= \max_{\overline{Q_T}}\max\{\bar{u},\underline{u}\}$, then $\bar{u}\geq\underline{u}$ in $Q_T$.
\end{lemma}

According to the assumption ${\bf{(F1)}}$, we choose a large constant $K$ satisfying
\begin{align}\label{3a_9}
|f(t,x,u_1)-f(t,x,u_2)|\leq K|u_1-u_2|.
\end{align}
For given  upper and lower solutions $\bar{u}(t,x),\underline{u}(t,x)$, by using $\underline{u}^{(0)}=\underline{u}$  and $\bar{u}^{(0)}=\bar{u}$ as the initial iterations, we can construct sequences $\{ \underline{u}^{(m)}\}^\infty_{m=1}$ and $\{\bar{u}^{(0)}\}^\infty_{m=1}$ such that the following equation
\begin{equation*}
\begin{cases}
        \bar{u}^{(m)}_t-dL\bar{u}^{(m)}+q\bar{u}^{(m)}_x+K\bar{u}^{(m)}=K\bar{u}^{(m-1)}+f(t,x,\bar{u}^{(m-1)}),&t\in(0,T_0),x\in{\Omega},\\
        \underline{u}^{(m)}_t-dL\underline{u}^{(m)}+q\underline{u}^{(m)}_x+K\underline{u}^{(m)}=K\underline{u}^{(m-1)}+f(t,x,\underline{u}^{(m-1)}),&t\in(0,T_0),x\in{\Omega},\\
         \bar{u}^{(m)}(t,l_1)=\underline{u}^{(m)}(t,l_1)=0,&t\in(0,T_0),\\
       \bar{u}^{(m)}(0,x)=\underline{u}^{(m)}(0,x)= u_0(x),&x\in{\Omega}
        \end{cases}
 \end{equation*}
 holds, where $Lu=\int_\Omega J(x-y)[u(t,y)-u(t,x)]dy$.
Based on the assumption ${\bf{(F1)}}$ and Theorem \ref{3theorem1}, it follows  that the limits of sequences $\{ \underline{u}^{(m)}\}^\infty_{m=1}$ and $\{\bar{u}^{(m)}\}^\infty_{m=1}$  exist and are unique in $\overline{Q_T}$. Then  we can apply a similar argument as in \cite[Theorem 3.6]{Wang}, combining with monotone dynamical system methods and Maximum principle (Lemma \ref{3lemma211}), to get the theorem as follows:
\begin{theorem}\label{3theorem2}
Assume that $\bf{(J), (J1), (F1)}$ and ${\bf{(F2)}}$ hold. For any given $u_0(x)$ satisfying \eqref{3a_2}, if \eqref{3a_11} has a bounded upper solution $\bar{u}(t,x)$ and a bounded lower solution $\underline{u}(t,x)$, then there exists $T>0$ such that
\eqref{3a_11} admits a unique local classical solution $u(t,x)\in C^1(\overline{Q_T})$.
\end{theorem}

We extend the local solution obtained in Theorem \ref{3theorem2} to the maximum time. To
do so, let $[0, T_0)$ be the maximum existence interval and $u(t,x)$ be the solution of \eqref{3a_11} defined on $[0, T_0)$.
\begin{theorem}\label{3theorem3}
Assume that $\bf{(J ),(J1), (F1)}$ and ${\bf{(F2)}}$ hold. For any given $u_0(x)$ satisfying \eqref{3a_2}, problem \eqref{3a_11} admits a unique classical solution $u(t,x)\in C^1([0, \infty )\times \bar{\Omega})$. Moreover,
\begin{align}
0< u(t,x)&\leq \max\{\max_{\bar{\Omega}} u_0(x),N\},\ \  \forall (t,x)\in {Q_\infty};\label{3a_a1}\\
& M_1  \leq u_x(t,x)\leq M_2,\ \ \forall (t,x)\in {Q_\infty};\label{3a_a2}
\end{align}
where $$M_1=\min\{\min_{\overline{\Omega}}u_0(x), \min_{\overline{\Omega}} u_0'(x)\}, M_2=\max\{\max_{\overline{\Omega}}u_0(x), \max_{\overline{\Omega}} u_0'(x)\},$$
and $M_1$ may be a negative constant and $M_2$ is a positive constant.
\end{theorem}
\begin{proof}
Firstly,  we shall show $T_0=\infty$. Obviously, $\underline{u_1}=0$ and $\bar{u}_1=\max\{\max_{\bar{\Omega}} u_0(x),N\}$ are respectively the lower and upper solutions, so \eqref{3a_11} admits the unique solution $u(t,x)$ and \eqref{3a_a1} holds on $(0,T_0)\times\bar{\Omega}$ by Lemma \ref{3lemma211}.


 Assume $T_0<\infty$. 
Then there is a sequence $\{T_n\}\subset(0,T_0)$ such that $T_n\to T_0$ as $n\to\infty$ and $\|u\|_{C^1(\overline{Q_{T_n}})}\leq \bar{u}_1$. Since $\bar{u}_1$ is independent of $T_n$ and $T_n\to T_0<\infty$, we have also $\|u\|_{C^1(\overline{Q_{T_0}})}\leq \bar{u}_1$. Furthermore, the solution of \eqref{3a_11} exists on $[0, T_0]$.  Regarding $T_n$ and $u(T_n,x)$ as
initial time and initial value. Similar to the proof of Theorem \ref{3theorem2},  we can find a constant $0<t_0\ll1$ such that \eqref{3a_11} has a unique solution $u_n$ in $[T_n, T_n+t_0]$. By the uniqueness, $u_n=u$ in $T_n\leq t<\min\{T_n+t_0,T_0\}$. it allows the solution to be extended beyond $T_0$, contradicting with the maximum existence interval. And so, \eqref{3a_a1} holds on $Q_\infty$.

We next shall prove \eqref{3a_a2}. Note that ${\bf{(F1)}}$ implies that there exists a constant $M>0$ such that
\[|f_x (t, x, u)|,|f_u (t, x,u)|\leq M,\]
and
\[\left|\frac{\partial f (t, x, u)}{\partial x}\right|=|f_x (t, x, u)+f_u (t, x,u)u_x|\leq M(1+|u_x|).\]

Denote $z(t,x)=(M_1-u_x)e^{-N_1t}$ for some positive constant $N_1$ and $z(0,x)=M_1-u_0'(x)\leq0$. We show $z(t,x)\leq0$ in $Q_\infty$ by contradiction. Suppose that we can find $(t_*,x_*)\in {Q_{\infty}}$ such that $z(t_*,x_*)=\sup_{{Q_{\infty}}} z(t,x)>0$. The choice of $M_1$ deduces $t_*>0$, and the bounbary condition means $x_*\in(l_1,l_2]$, $z_x(t_*,x_*)\geq0$ and $z_t(t_*,x_*)=0$. In view of the assumption ${\bf{(J1)}}$, we know  that $z(t_*,x_*)$ satisfies
\begin{align*}
0= z_t(t_*,x_*)=&d\int_\Omega  J(x_*-y)[z(t_*,y)-z(t_*,x_*)]dy-qz_x(t_*,x_*)+(f_u(t_*,x_*,u(t_*,x_*))-N_1)z(t_*,x_*)\\
&-[f_x(t_*,x_*,u(t_*,x_*))+f_u(t_*,x_*,u(t_*,x_*))M_1+m(t_*,x_*)]e^{-N_1t_*}\\
\leq&(f_u(t_*,x_*,u(t_*,x_*))-N_1)z(t_*,x_*)+(|f_x (t, x, u)|+|f_u (t, x, u)|M_1+|m(t_*,x_*)|)e^{-N_1t_*}\\
\leq &(M-N_1)z(t_*,x_*)+(M+MM_1+\|m(t,x)\|_{\infty})e^{-N_1t_*}<0,
\end{align*}
provided that $N_1$ is large enough satisfying
\[(M-N_1)z(t_*,x_*)+(M+MM_1+\|m(t,x)\|_{\infty})e^{-N_1t_*}<0,\]
where $m(t,x)$ is given by \eqref{3a_29}.
This is a contradiction, and then $z(t,x)=(M_1-u_x)e^{-N_1t}\leq0$ on $Q_\infty$. That is to say,
\[ M_1  \leq u_x(t,x),\ \ \forall (t,x)\in {Q_\infty}\]
Similarly, we can define $w(t,x)=(M_2-u_x)e^{-N_2t}$ for some large enough positive constant $N_2$ and $z(0,x)=M_2-u_0'(x)\geq0$ to obtain $w(t,x)\geq0$ on $Q_\infty$. So,
\[ M_2  \geq u_x(t,x),\ \ \forall (t,x)\in {Q_\infty}.\]
The proof is completed.
\end{proof}

 And we can use the same methods to consider the nonlocal Dirichlet problem \eqref{3a_1}.  On the one hand, we can establish the maximum principle for \eqref{3a_1}:
\begin{lemma}[Maximum principle]\label{3lemma2}
For any given $u_0(x)\geq0$ and $T>0$, assume that, for some $c(t,x)\in L^\infty(\overline{Q_T})$, function $u(t,x)\in C^1(\overline{Q_T})$ satisfies
\begin{equation*}
\begin{cases}
        u_t\geq d[\int_\Omega J(x-y)u(t,y)dy-u(t,x)]-qu_x+c(t,x)u,&t\in(0,T),x\in\Omega,\\
         u(t,l_1)\geq0,&t\in(0,T),\\
        u(0,x)=u_0(x)\geq0,&x\in\Omega.
        \end{cases}
 \end{equation*}
Then $u(t,x)\geq0$ in $\overline{Q_T}$. Moreover, $u(t,x)>0$ in $(0,T]\times(l_1,l_2]$ if $u_0(x)\not\equiv0$ in $\bar{\Omega}$.
\end{lemma}
 On the other hand, the well-posedness of solutions of \eqref{3a_1} is obtained as follows:
 \begin{theorem}\label{3theorem4}
Assume that $\bf{(J ),(J1), (F1)}$ and ${\bf{(F2)}}$ hold. For any given $u_0(x)$ satisfying \eqref{3a_2},
\eqref{3a_1} admits a unique classical solution $u(t,x)\in C^1([0, \infty )\times \bar{\Omega})$ satisfying \eqref{3a_a1} and \eqref{3a_a2}.
\end{theorem}

\section{Nontrivial stationary solution \label{3a_part3}}
\setcounter{equation}{0}
In this section, we focus on the specific reaction term.  Assume further that  $f(t,x,u)=uh(x,u)$  satisfies the following conditions:
\begin{enumerate}[(1)]
 \item  Function $h(x,u)\in C^{1}(\bar{\Omega}\times\mathbb{R}^+)$ is decreasing in $u$ on $\mathbb{R}^+$;
 \item Function $h(x,0)$ may be sign changed.
\end{enumerate}
 Then \eqref{3a_11} and \eqref{3a_1} become a nonlocal reaction-diffusion-advection Fisher-KPP equation.  For given  initial value $u_0(x)$, it follows that  \eqref{3a_11} (or \eqref{3a_1}) admits a unique solution $u(t,x;u_0)$ satisfying $u(t,x;u_0)>0$ for all $t>0$.

Moreover, if the nontrivial stationary $U(x)$ of \eqref{3a_11} (or \eqref{3a_1}) exists, then
we say that  $U(x)$ is stable: if for any $\epsilon>0$, there is $\delta>0$ such that when $\|u_0-U(x)\|_\infty<\delta$,  we have $\|u(t,x;u_0)-U(x)\|_\infty<\epsilon$ for all $t>0$. $U(x)$ is unstable if this not holds. $U(x)$ is asymptotically stable if for any $u_0(x)$, we have $\|u(t,x;u_0)-U(x)\|_\infty\to0$ as $t\to\infty$.

\subsection{Spectral theory of integro-differential operator}\label{3subsection1}
The eigenvalue problem is an important tool to study the longtime behaviors of elliptic and parabolic problems, especially the sign of principal eigenvalue of the linearized operator can determine  the existence and stability of nontrivial stationary solution, even the persistence of species, see \cite{Berestycki2016,DeLeenheer,Covill2010,Shen2015}.

Now, we recall some known results on the principal eigenvalue of the linear nonlocal operator $\mathfrak{L}_{\Omega,J,q}+a$, 
which is defined on $C^1(\bar{\Omega})$ by
\[(\mathfrak{L}_{\Omega,J,q}+a)[\varphi(x)]:=d\int_\Omega J(x-y)\varphi(y)dy-q\varphi'(x)+a(x)\varphi(x),\]
where $a(x)\in C(\bar{\Omega})\cap L^\infty(\Omega)$.
Consider the following eigenvalue problem:
\begin{equation}\label{3a_12}
d\int_\Omega J(x-y)\varphi(y)dy-q\varphi'(x)+a(x)\varphi(x)+\lambda\varphi(x)=0,\ \ \ x\in\Omega.
\end{equation}
The eigenvalue is called the principal eigenvalue if there is a continuous and positive eigenfunction such that \eqref{3a_12} holds. However, due to the lack of regularity and compactness of nonlocal dispersal operator, \eqref{3a_12} may not admit principal eigenvalue. In \cite{DeLeenheer}, when $\Omega=\mathbb{R}$ and $a(x)$ is a periodic function, the Krein-Rutman Theory can be used to show the existence of principal eigenvalue of \eqref{3a_12}. 
As shown in \cite{Coville2020,Coville2021}, if the zero order term $a(x)$ is just assumed bounded, the generalized principal eigenvalue can be defined to characterise the principal eigenvalue of operator $\mathfrak{L}_{\Omega,J,q}+a$:
\begin{equation}\label{3a_13}
\begin{aligned}
\lambda_p(\mathfrak{L}_{\Omega,J,q}+a):=\sup\{\lambda\in\mathbb{R}|\exists \varphi\in &C^1(\Omega)\cap C(\bar{\Omega}) \ \mbox{satisfying} \ \varphi>0 \ \mbox{in} \ \Omega \  \mbox{and}
\\
&\mathfrak{L}_{\Omega,J,q}[\varphi(x)]+a(x)\varphi(x)+\lambda\varphi(x)\leq0\}.
\end{aligned}
\end{equation}
When $q = 0$, the generalized principal eigenvalue can be defined as
\begin{equation*}
\begin{aligned}
\lambda_p(\mathfrak{L}_{\Omega,J}+a):=\sup\{\lambda\in\mathbb{R}|\exists \varphi\in & C(\bar{\Omega}) \ \mbox{satisfying} \ \varphi>0 \ \mbox{in} \ \Omega \  \mbox{and}
\\
&\mathfrak{L}_{\Omega,J}[\varphi(x)]+a(x)\varphi(x)+\lambda\varphi(x)\leq0\}.
\end{aligned}
\end{equation*}
And some standard properties have shown in \cite{Berestycki2016}.

\begin{proposition}\label{3proposition1}
Assume that $\bf{(J )}$  holds. Let $\Omega\subset\mathbb{R}$ be a non-empty open interval, possibly unbounded. For $q\neq0$, the following conclusions are valid:
\begin{enumerate}[(1)]
 \item If $\Omega_1\subset\Omega_2$ are non-empty open domains, then $\lambda_p(\mathfrak{L}_{\Omega_1,J,q}+a)\geq\lambda_p(\mathfrak{L}_{\Omega_2,J,q}+a)$.
 \item If $a_1(x),a_2(x)\in C(\bar{\Omega})\cap L^\infty(\Omega)$ satisfying $a_1(x)\leq a_2(x)$, then $$\lambda_p(\mathfrak{L}_{\Omega,J,q}+a_1)\geq\lambda_p(\mathfrak{L}_{\Omega,J,q}+a_2).$$
 \item If $J$ further has a compact support and satisfies {\bf{(J1)}}, then
 $$\lambda_p(\mathfrak{L}_{\Omega,J,q}+a)=\lambda_p(\mathfrak{L}_{\Omega,J,-q}+a).$$
 \item If $\Omega$ is bounded, then $\lambda_p(\mathfrak{L}_{\Omega,J,q}+a)$ is continuous in $q\in(-\infty,0)\cup(0,+\infty)$.
\end{enumerate}
\end{proposition}
\begin{proof}
We refer to \cite[Proposition 1.3]{Coville2020}, \cite[Proposition 3.4]{Coville2021} and \cite[Theorem 6.2]{DeLeenheer} for the proofs of (1)-(4) when $q\neq0$.
\end{proof}

\begin{proposition}[\cite{Coville2020}, Theorem 1.2]\label{3proposition2}
Assume that $\bf{(J )}$  holds. Let $\Omega=(l_1,l_2)$ be a non-empty bounded open interval. For $q\neq0$, we have
\[\lambda_p(\mathfrak{L}_{\Omega,J,q}+a)=\lambda_p'(\mathfrak{L}_{\Omega,J,q}+a).\]
If $q<0$, then
\begin{align*}
\lambda_p'(\mathfrak{L}_{\Omega,J,q}+a):=\inf\{\lambda\in\mathbb{R}|\exists \varphi\in &C^1(\Omega)\cap C(\bar{\Omega}) \ \mbox{satisfying} \ \varphi>0 \ \mbox{in} \ \Omega,\varphi(l_2)=0,\varphi(l_1)>0, \\  &\mbox{and} \
\mathfrak{L}_{\Omega,J,q}[\varphi(x)]+a(x)\varphi(x)+\lambda\varphi(x)\geq0\}.
\end{align*}
If $q>0$, then
\begin{align*}
\lambda_p'(\mathfrak{L}_{\Omega,J,q}+a):=\inf\{\lambda\in\mathbb{R}|\exists \varphi\in &C^1(\Omega)\cap C(\bar{\Omega}) \ \mbox{satisfying} \ \varphi>0 \ \mbox{in} \ \Omega,\varphi(l_1)=0,\varphi(l_2)>0, \\  &\mbox{and} \
\mathfrak{L}_{\Omega,J,q}[\varphi(x)]+a(x)\varphi(x)+\lambda\varphi(x)\geq0\}.
\end{align*}
Moreover,  the infimum can be achieved. That is to say, there is a function $\varphi\in C^1(\Omega)\cap C(\bar{\Omega})$ such that \eqref{3a_12} holds and $\varphi>0$  in $\Omega$  satisfying  $\varphi(l_2)=0$ for $q<0$ and $\varphi(l_1)=0$ for $q>0$.
\end{proposition}

\begin{proposition}\label{3proposition3}
 Assume that $a(x)\in C(\bar{\Omega})\cap L^\infty(\Omega)$ is nonnegative, $\Omega=(l_1,l_2)$ is a non-empty bounded open interval and $\bf{(J), (J1)}$  holds.  If further $J$ has a compact support, then for fixed $\Omega$, $$\lambda_p(\mathfrak{L}_{\Omega,J,q_1}+a)<\lambda_p(\mathfrak{L}_{\Omega,J,q_2}+a), \ \ 0<q_1< q_2;$$
      $$\lambda_p(\mathfrak{L}_{\Omega,J,q_1}+a)>\lambda_p(\mathfrak{L}_{\Omega,J,q_2}+a), \ \ q_1< q_2<0.$$
\end{proposition}
\begin{proof}
By Proposition \ref{3proposition1}(3), we just discuss the case for $q>0$.  Note that the eigenvalue problem
\eqref{3a_12} admits a principal eigen pair $(\lambda_p(\mathfrak{L}_{\Omega,J,q}+a),\varphi(x))$ satisfying $\varphi(x)>0$ in $\Omega$ and $\varphi(l_1)=0,\varphi(l_2)>0$.

{\bf{Claim}}: the principal eigenfunction $\varphi(x)$ is nondecreasing.
If the claim is not true, then there exists $x_0\in\Omega$ such that $\varphi'(x_0)<0$. Notice that
\[d\int_\Omega J(x_0-y)\varphi(y)dy-q\varphi'(x_0)+a(x_0)\varphi(x_0)+\lambda_p(\mathfrak{L}_{\Omega,J,q}+a)\varphi(x_0)=0,\ \ \ x\in\Omega.\]
By the nonnegativity of $a(x)$ and the positivity of $\varphi(x)$, the principal eigenvalue $\lambda_p(\mathfrak{L}_{\Omega,J,q}+a)<0$ for any $q>0$. However, since $\varphi\in C^1(\Omega)\cap C(\bar{\Omega})$ and $\varphi(l_1)=0,\varphi(l_2)>0$, we can choose a nonempty set $D\subset\Omega$ and a constant $\delta>0$ such that $\varphi'(x)\geq \delta$ in $D$. In view of $a(x)\in C(\bar{\Omega})\cap L^\infty(\Omega)$, we can find a constant $C>0$ such that $d\int_\Omega J(x-y)\varphi(y)dy+a(x)\varphi(x)\leq C$ for $x\in D$. When $q$ is large enough satisfying $q>\frac{C}{\delta}$, there holds
\[-\lambda_p(\mathfrak{L}_{\Omega,J,q}+a)\varphi(x)=d\int_\Omega J(x-y)\varphi(y)dy-q\varphi'(x)+a(x)\varphi(x)\leq C-q\varphi'(x)\leq 0,x\in D\]
which is a contradiction with $\lambda_p(\mathfrak{L}_{\Omega,J,q}+a)<0$ for any $q>0$. The claim is true.

For $-q_2<0<q_1< q_2$, Proposition \ref{3proposition2} means that there exist $\phi>0$ satisfying $\phi(l_1)=0,\phi(l_2)>0$, $\phi'\geq0$ and $\psi>0$ satisfying $\psi(l_2)=0,\psi(l_1)>0$, $\psi'\leq0$ such that
\begin{align}
d\int_\Omega J(x-y)\phi(y)dy-q_1\phi'(x)+a(x)\phi(x)+\lambda_p(\mathfrak{L}_{\Omega,J,q_1}+a)\phi(x)=0,\ \ \ x\in\Omega;\label{3a_17}\\
d\int_\Omega J(x-y)\psi(y)dy+q_2\psi'(x)+a(x)\psi(x)+\lambda_p(\mathfrak{L}_{\Omega,J,-q_2}+a)\psi(x)=0,\ \ \ x\in\Omega.\label{3a_19}
\end{align}
Multiplying \eqref{3a_17} by $\psi$ and \eqref{3a_19} by $\phi$, and then, integrating both sides in $\Omega$ and subtracting, we obtain
\begin{align*}
(\lambda_p(\mathfrak{L}_{\Omega,J,q_1}+a)-\lambda_p(\mathfrak{L}_{\Omega,J,-q_2}+a))\int_\Omega \phi\psi dx&=\int_\Omega (q_2\phi\psi'+q_1\phi'\psi)dx\\
&=q_2(\phi\psi)\mid_{l_1}^{l_2}+(q_1-q_2)\int_\Omega\phi'\psi dx<0
\end{align*}
By Proposition \ref{3proposition1}(3),  $\lambda_p(\mathfrak{L}_{\Omega,J,q_1}+a)<\lambda_p(\mathfrak{L}_{\Omega,J,-q_2}+a)=\lambda_p(\mathfrak{L}_{\Omega,J,q_2}+a)$. The proof is completed.

\end{proof}
\begin{remark}\label{3remark1}
The following nonlocal eigenvalue problem
\begin{equation*}
d\int_\Omega J(x-y)[\varphi(y)-\varphi(x)]dy-q\varphi'(x)+b(x)\varphi(x)+\lambda\varphi(x)=0,\ \ \ x\in\Omega.
\end{equation*}
can be written as \eqref{3a_12} as long as $a(x):=b(x)-d\int_\Omega J(x-y)dy$.
\end{remark}

\subsection{The existence, uniqueness and stability of nontrivial stationary solution}
Firstly, we shall the strong maximal principle for the stationary problem.
\begin{lemma}[Strong maximum principle]\label{3lemma5}
For any given $T>0$, assume that, for some $c(x)\in L^\infty(\Omega)$ and $q>0$, function $u(x)\in C^1(\bar{\Omega})$ satisfies $u(x)\geq 0$ in $\bar{\Omega}$ and
\begin{equation*}
\begin{cases}
        d\int_\Omega J(x-y)[u(y)-u(x)]dy-qu'(x)+c(x)u(x)\leq 0,&x\in\Omega,\\
         u(l_1)\geq0.
        \end{cases}
 \end{equation*}
Then $u(x)>0$ in $(l_1,l_2]$.
\end{lemma}
\begin{proof}
On the contrary, we assume that there exists $x_0\in (l_1,l_2]$ such that $u(x_0)=0$. If $x_0\in\Omega$, then we have $u'(x_0)=0$, and so
\[d\int_\Omega J(x_0-y)u(y)dy= d[\int_\Omega J(x_0-y)u(y)-u(x_0)]dy-qu'(x_0)+c(x_0)u(x_0)\leq0,\]
which is impossible. If $x_0=l_2$, then $u(x)>0$ in $(l_1,l_2)$, namely, $u'(l_2)\leq0$. Then
\[d\int_\Omega J(l_2-y)u(y)dy-qu'(l_2)= d[\int_\Omega J(l_2-y)u(y)-u(l_2)]dy-qu'(l_2)+c(l_2)u(l_2)\leq0,\]
which is a contraction due to $q>0$. The proof is completed.
\end{proof}

\begin{remark}
For any given $T>0$, assume that, for some $c(x)\in L^\infty(\Omega)$ and $q<0$, function $u(x)\in C^1(\bar{\Omega})$ satisfies $u(x)\geq 0$ in $\bar{\Omega}$ and
\begin{equation*}
\begin{cases}
        d\int_\Omega J(x-y)[u(y)-u(x)]dy-qu'(x)+c(x)u(x)\leq 0,&x\in\Omega,\\
         u(l_2)\geq0.
        \end{cases}
 \end{equation*}
Then $u(x)>0$ in $[l_1,l_2)$.
\end{remark}

In this subsection, the existence of nontrivial stationary solutions for \eqref{3a_11} or \eqref{3a_1} will be proved by the upper-lower solution methods.  For the simplicity of discussion, we consider
\begin{equation}
\label{3a_14}
\begin{cases}
        d\int_\Omega J(x-y)[u(y)-u(x)]dy-qu_x+uh(x,u)=0,&x\in{\Omega},\\
        u(l_1)=0,\\
       0 \leq u(x) \leq\max\{\max_{\bar{\Omega}} u_0(x),N\},&x\in{\Omega}
        \end{cases}
 \end{equation}
and
\begin{equation}
\label{3a_15}
\begin{cases}
        d[\int_\Omega J(x-y)u(y)dy-u]-qu_x+uh(x,u)=0,&x\in\Omega,\\
         u(l_1)=0,\\
         0 \leq u(x) \leq\max\{\max_{\bar{\Omega}} u_0(x),N\},&x\in\Omega
        \end{cases}
 \end{equation}
Notice that  \eqref{3a_14} (or  \eqref{3a_15}) has only one trivial solution $u=0$.
In the following, we still only consider the details for \eqref{3a_14}.
\begin{theorem}\label{3theorem5}
 Problem \eqref{3a_14} has a continuous and bounded nontrivial solution, provided that
\[\lambda_p(\mathfrak{L}_{\Omega,J,q}+h(x,0)-d\int_\Omega J(x-y)dy)<0.\]
\end{theorem}
\begin{proof}
From subsection \ref{3subsection1},  we know that the eigenvalue problem
\begin{equation}\label{3a_16}
d\int_\Omega J(x-y)\varphi(y)dy-q\varphi'(x)+[h(x,0)-d\int_\Omega J(x-y)dy]\varphi(x)+\lambda\varphi(x)=0,\ \ \ x\in\Omega
\end{equation}
admits a principal eigenvalue $\lambda_p(\mathfrak{L}_{\Omega,J,q}+h(x,0)-d\int_\Omega J(x-y)dy)$ associated with  eigenfunction $\varphi(x)\in C^1(\Omega)\cap C(\bar{\Omega})$ satisfying $\varphi(x)>0$ in $(l_1,l_2]$ and $\varphi(l_1)=0$. Let $\underline{u}(x)=\epsilon\varphi(x)$ in $\bar{\Omega}$, where $\epsilon>0$ is a sufficiently small  constant satisfying $\lambda_p(\mathfrak{L}_{\Omega,J,q}+h(x,0)-d\int_\Omega J(x-y)dy)+h(x,0)-h(x,\epsilon\varphi)\leq0$. A direct calculation gives that
\begin{align*}
&d\int_\Omega J(x-y)[\underline{u}(y)-\underline{u}(x)]dy-q\underline{u}_x+h(x,\underline{u})\underline{u}\\
=&\epsilon[d\int_\Omega J(x-y)[\varphi(y)-\varphi(x)]dy-q\varphi'(x)+h(x,0)\varphi(x)]-[h(x,0)-h(x,\epsilon\varphi)]\epsilon\varphi\\
=&\epsilon\varphi[-\lambda_p(\mathfrak{L}_{\Omega,J,q}+h(x,0)-d\int_\Omega J(x-y)dy)-h(x,0)+h(x,\epsilon\varphi)]\geq0,
\end{align*}
which means that $\underline{u}$ is a lower solution of \eqref{3a_14}. Take a large constant $M>\max\{\max_{\bar{\Omega}} u_0(x),N\}$ as the upper solution $\bar{u}=M$, then $\underline{u}\leq\bar{u}$. Denote $u(t,x;\bar{u})$ as the solution of \eqref{3a_11} with initial value $\bar{u}$. Lemma \ref{3lemma211} means that
\[\bar{u}\geq u(t_2,x;\underline{u})=u(t_1,x;u(t_2-t_1,\cdot;\underline{u}))\geq u(t_1,x;\underline{u})\geq\underline{u},\ \forall 0<t_1<t_2,x\in\bar{\Omega},\]
and
\[\underline{u}\leq u(t_2,x;\bar{u})=u(t_1,x;u(t_2-t_1,\cdot;\bar{u}))\leq  u(t_1,x;\bar{u})\leq\bar{u},\ \forall 0<t_1<t_2,x\in\bar{\Omega}.\]
 Since $u(t,x;\underline{u})$ is increasing in $t$ and uniformly bounded, there exists $U_1(x)\in C(\bar{\Omega})$ such that $\lim\limits_{t\to\infty} u(t,x;\underline{u})=U_1(x)$ for any given $x\in\bar{\Omega}$. Noting that $u(t,x;\underline{u})$ is continuous in $x\in\bar{\Omega}$(Lemma \ref{3lemma4}), by Dini theorem, $u(t,x;\underline{u})$ converges uniformly to $U_1(x)$ as $t\to\infty$. Similarly, there exists $U_2(x)\in C(\bar{\Omega})$ such that $\lim\limits_{t\to\infty} u(t,x;\bar{u})=U_2(x)$ uniformly on $\bar{\Omega}$. Namely, \eqref{3a_14} has a continuous and bounded nontrivial solution. This completes the proof.
\end{proof}

\begin{theorem}
Any nontrivial  solution $U(x)$  of \eqref{3a_14} satisfies $U(x)\in C^1(\bar{\Omega})$ and $U(x)>0$ in $(l_1,l_2]$.
\end{theorem}
\begin{proof}
Obviously, $0\leq U(x) \leq \max\{\max_{\bar{\Omega}} u_0(x),N\}$.  Note that we can find a positive constant $K>\|h(x,0)\|_\infty$ such that
\[ d\int_\Omega J(x-y)[U(y)-U(x)]dy-qU_x-KU=[-K-h(x,U)]U\leq (\|h(x,0)\|_\infty-K) U\leq0.\]
It follows from Lemma \ref{3lemma5} that $U(x)>0$ in $(l_1,l_2]$ or $U(x)\equiv0$ in $\bar{\Omega}$. If there is $x_0\in(l_1,l_2]$ satisfying $U(x_0)=0$, then  it deduces $U(x)\equiv0$ in $\bar{\Omega}$, which is a contradiction with the definition of nontrivial solution. And so, $U(x)>0$ in $(l_1,l_2]$.

 In addition, similar to the argument of \cite{Bates2007}, we prove $U_1(x),U_2(x)\in C^1(\bar{\Omega})$. Without loss of generality, we only prove $U_1(x)\in C^1(\bar{\Omega})$. Defining $w=U_1(x_1)-U_1(x_2)$, we know
\begin{align*}
|q||U_1'(x_1)-U_1'(x_2)|\leq&d\int_\Omega |J(x_1-y)-J(x_2-y)|U_1(y)dy+|h(x_1,U_1(x_1))-d\int_\Omega J(x_1-y)dy||w|\\
&+d\int_\Omega |J(x_1-y)-J(x_2-y)|dyU_1(x_2)+|h(x_1,U_1(x_1))-h(x_1,U_1(x_2))|U_1(x_2)\\
&+|h(x_1,U_1(x_2))-h(x_2,U_1(x_2))|U_1(x_2)\\
\leq& Md\int_\Omega |J(x_1-y)-J(x_2-y)|dy+|h(x_1,U_1(x_1))-d\int_\Omega J(x_1-y)dy||w|\\
&+M|h(x_1,U_1(x_1))-h(x_1,U_1(x_2))|+M|h(x_1,U_1(x_2))-h(x_2,U_1(x_2))|,
\end{align*}
which means $U_1'(x)\in C(\bar{\Omega})$ since function $h(x,u)\in C^1(\bar{\Omega}\times\mathbb{R}^+)$, $J:\mathbb{R}\to\mathbb{R}^+$ is continuous and $U_1(x)\in C(\bar{\Omega})$. The proof is completed.
\end{proof}

\begin{theorem}\label{3theorem6}
The nontrivial solution  of \eqref{3a_14} is unique.
\end{theorem}
\begin{proof}
To do so assume that there are two different solutions $u_1(x),u_2(x)\in C^1(\bar{\Omega})$ satisfying $u_1(l_1)=u_2(l_1)=0$. Set $v(x)=\sup\{u_1(x),u_2(x)\}$ and $b(x)=h(x,v)$.
Define
\begin{align*}
\lambda_p''(\mathfrak{L}_{\Omega,J,-q}&-d\int_\Omega J(x-y)dy+b):=\\
\inf\{\lambda\in\mathbb{R}&|\exists \varphi\in C^1(\Omega)\cap C(\bar{\Omega}) \ \mbox{satisfying} \ \varphi>0 \ \mbox{in} \ [l_1,l_2), \\  &\varphi(l_2)=0, \ \mbox{and} \
\mathfrak{L}_{\Omega,J,-q}[\varphi(x)]-d\int_\Omega J(x-y)dy\varphi(x)+b(x)\varphi(x)+\lambda\varphi(x)\geq0\}.
\end{align*}
By \cite[Proposition 6.1]{Coville2021}, we know  $\lambda_p''(\mathfrak{L}_{\Omega,J,-q}-d\int_\Omega J(x-y)dy+b)\leq0$. If $\lambda_p''(\mathfrak{L}_{\Omega,J,-q}-d\int_\Omega J(x-y)dy+b)=0$, we can choose a  monotone decreasing sequence $\{\lambda_n\}_{n\in \mathbb{N}}$ of positive numbers such that $|\lambda_n|<1$ for all $n\in\mathbb{N}$ and $\lambda_n\to0$ as $n\to\infty$ and choose a nonnegative function sequence $(\varphi_n)_{n\in\mathbb{N}}$ such that $\varphi_n(l_2)=0$, $\|\varphi_n\|\leq  N$ such that there holds
\[\mathfrak{L}_{\Omega,J,-q}[\varphi_n(x)]-d\int_\Omega J(x-y)dy\varphi_n(x)+b(x)\varphi_n(x)+\lambda_n\varphi_n(x)\geq0, \ \forall \ x\in\Omega.\]
Then we have $0\leq \varphi_n(x)\leq w:=\max\{\max_{\bar{\Omega}} u_0(x),N\}$, and so $(\varphi_n)_{n\in\mathbb{N}}$ is uniformly bounded in $C([l_1,l_2])$. Obviously, $\inf_{\bar{\Omega}}\varphi'_n(x)$ is achieved at some point $x_1^n$ on $\bar{\Omega}$. Then there exists a large constant $M>0$ such that
\begin{align*}
q\inf_{\bar{\Omega}}\varphi'_n(x)&=q\varphi'_n(x_1^n)\\
&\geq -d\int_{\Omega} J(x_1^n-y)\varphi_n(y)dy+d\int_\Omega J(x_1^n-y)dy\varphi_n(x_1^n)-b(x_1^n)\varphi_n(x_1^n)-\lambda_n\varphi_n(x_1^n)\\
&\geq -2dw-|\lambda_n|w -(\sup_{u\in[0,w]}h(x,u))w\\
&\geq -2dw-w -(\sup_{u\in[0,w]}h(x,u))w:=-M.
\end{align*}
By $\varphi_n\geq0$ and $\varphi_n(l_2)=0$, we deduce that
$(\varphi_n)_{n\in\mathbb{N}}$ is also uniformly bounded in $C^1(\bar{\Omega})$. Therefore, by diagonal extraction process, we can find a converging subsequence such that there exists $\varphi\geq,\not\equiv0$ such that $\varphi_n\to\varphi$ in
$C^1_{loc}({\Omega})$.  Moreover,  $\varphi\in L^\infty(\bar{\Omega})$, which satisfies $\varphi(l_2)=0$ and
\[\mathfrak{L}_{\Omega,J,-q}[\varphi(x)]-d\int_\Omega J(x-y)dy\varphi(x)+b(x)\varphi(x)=0, \ \forall \ x\in\Omega.\]
Then, there exist sequences $(\lambda_n)_{n\in\mathbb{N}}$ and $(\varphi_n)_{n\in\mathbb{N}}$ and a nonnegative nontrivial function $\varphi\in L^\infty(\bar{\Omega})$ such that
\begin{enumerate}[(1)]
 \item  $\lambda_n>0$, $(\lambda_n)_{n\in\mathbb{N}}$ is a monotone decreasing sequence and converges to $0$.
 \item  $(\varphi_n)_{n\in\mathbb{N}}>0$ in $(l_1,l_2]$ and $\varphi_n\in C^1(\Omega)\cap C(\bar{\Omega})$ converges to $\varphi$ in $C^1(\bar{\Omega})$ for any $x\in \bar{\Omega}$.
 \item   $\varphi_n(l_2)=\varphi(l_2)=0$ and
\begin{align}\label{3a_28}
\mathfrak{L}_{\Omega,J,-q}[\varphi_n(x)]-d\int_\Omega J(x-y)dy\varphi_n(x)+b(x)\varphi_n(x)+\lambda_n\varphi_n(x)\geq0, \ \forall \ x\in\Omega.
\end{align}
\end{enumerate}
If $\lambda_p''(\mathfrak{L}_{\Omega,J,-q}-d\int_\Omega J(x-y)dy+b)<0$, it follows from the definition of $\lambda_p''$ that there is $\varphi>0$ in $[l_1,l_2)$ satisfying $\varphi\in C^1(\Omega)\cap C(\bar{\Omega})$ such that $\varphi(l_2)=0$ and
\[\mathfrak{L}_{\Omega,J,-q}[\varphi(x)]-d\int_\Omega J(x-y)dy\varphi(x)+b(x)\varphi(x)=-\lambda_p''\varphi(x)\geq0,  \ \forall \ x\in\Omega.\]
We can obtain $(1)-(3)$ by
taking the sequence $(\varphi_n)_{n\in\mathbb{N}}$ defined for all $n$ by $\varphi_n=\varphi$, and any decreasing sequence $(\lambda_n)_{n\in\mathbb{N}}$ that converges to $0$.

We multiply \eqref{3a_28} by $u_1(x)$ and integrate it over $\Omega$, and then
\[\mathcal{R}_n:=d\int_\Omega \int_\Omega J(x-y)(\varphi_n(y)-\varphi_n(x))dyu_1(x)dx+q\int_\Omega\varphi_n'(x)u_1(x)dx+\int_\Omega(b(x)+\lambda_n)
\varphi_n(x)u_1(x)dx\geq0.\]
Thus,
\[\liminf\limits_{n\to\infty}\mathcal{R}_n\geq0.\]
Notice that
\begin{align*}
\mathcal{R}_n=&d\int_\Omega \int_\Omega J(x-y)(u_1(y)-u_1(x))dy\varphi_n(x)dx-q\int_\Omega\varphi_n(x)u_1'(x)dx+\varphi_n(x)u_1(x)|_{l_1}^{l_2}\\
&+\int_\Omega(b(x)+\lambda_n)
\varphi_n(x)u_1(x)dx\\
&=\int_\Omega(b(x)-h(x,u_1(x)))\varphi_n(x)u_1(x)dx+\int_\Omega\lambda_n\varphi_n(x)u_1(x)dx.
\end{align*}
Since $h(x,u)$ is decreasing in $u$ on $\mathbb{R}^+$ and $v(x)\geq,\not\equiv u_1(x)$, we get
\[\limsup\limits_{n\to\infty}\mathcal{R}_n=\int_\Omega [h(x,v(x))-h(x,u_1(x))]\varphi(x)u_1(x)dx<0.\]
Therefore,
\[0\leq \liminf\limits_{n\to\infty}\mathcal{R}_n\leq\limsup\limits_{n\to\infty}\mathcal{R}_n<0.\]
We get a contradiction, and so $u_1(x)\equiv u_2(x)$. The proof is completed.
\end{proof}

\begin{theorem}\label{3theorem7}
Assume that $U(x)\in C^1(\bar{\Omega})$ is the unique nontrivial solution  of \eqref{3a_14}. Then $U(x)$ is asymptotically stable.
\end{theorem}
\begin{proof}
 From Theorem \ref{3theorem5}, we can choose the upper  solution $\bar{u}=\max\{\max_{\bar{\Omega}} u_0(x),N\}$ and lower solution $\underline{u}=\epsilon\varphi$, where $\varphi$ is the corresponding principal eigenfunction with principal eigenvalue $\lambda_p(\mathfrak{L}_{\Omega,J,q}+h(x,0)-d\int_\Omega J(x-y)dy)$ and $\epsilon$ is small enough such that $\epsilon\varphi\leq u(1,x;u_0)$, where $u(1,x;u_0)>0$. And then, for any given $u_0(x)\in C^1(\bar{\Omega})$, the  comparison argument deduces
\[\underline{u}\leq u(t,x;\underline{u})\leq u(t,x;u(1,x;u_0))\leq u(t,x;\bar{u})\leq\bar{u}, \forall t>1.\]
From the fact that $u(t,x;\underline{u})$ is increasing in $t>1$ and $u(t,x;\bar{u})$ is decreasing in $t>1$, there exist $U_1(x),U_2(x)$
 such that $\lim\limits_{t\to\infty} u(t,x;\underline{u})=U_1(x)$ and $\lim\limits_{t\to\infty} u(t,x;\bar{u})=U_2(x)$ uniformly on $\bar{\Omega}$.  It follows from Theorem \ref{3theorem6} that $U_1(x)=U_2(x)=U(x)$ and $\lim\limits_{t\to\infty} u(t,x;u_0)=U(x)$ uniformly on $\bar{\Omega}$. The proof is completed.
\end{proof}

\begin{lemma}\label{3lemma3}
Denote by $U_0=0$  the unique trivial solution of  \eqref{3a_14}.
\begin{enumerate}[(1)]
 \item If $\lambda_p(\mathfrak{L}_{\Omega,J,q}+h(x,0)-d\int_\Omega J(x-y)dy)<0$, then $U_0$ is unstable.
 \item If $\lambda_p(\mathfrak{L}_{\Omega,J,q}+h(x,0)-d\int_\Omega J(x-y)dy)>0$,  then $U_0$  is asymptotically stable.
\end{enumerate}
\end{lemma}
\begin{proof}
(1)  Assume that $U_0$ is stable. That is to say that for any $\epsilon>0$ and small enough $\delta>0$,  given initial value  $u_0(x)\in[0,\delta)$ in $\bar{\Omega}$, we have $\|u(t,x;u_0)\|_\infty<\epsilon$ for all $t>0$.  Obviously, by assumption {\bf{(F1)}}, for any $0\leq u<\epsilon$, there holds
$$f(x,u)\geq (h(x,0)+\frac{\lambda_p(\mathfrak{L}_{\Omega,J,-q}+h(x,0)-d\int_\Omega J(x-y)dy)}{2})u.$$
Then  $u(t,x;u_0)>0$ in $(0,+\infty)\times(l_1,l_2]$ by maximum principle. And the solution of \eqref{3a_11} can be written as
\[u(t,x;u_0)=u_0(x)+\int_0^t\{d\int_\Omega J(x-y)[u(s,y)-u(s,x)]dy-qu_x(s,x)+f(x,u(s,x))\}ds.\]
By Proposition \ref{3proposition1}(3), $$\lambda_p(\mathfrak{L}_{\Omega,J,-q}+h(x,0)-d\int_\Omega J(x-y)dy)=\lambda_p(\mathfrak{L}_{\Omega,J,q}+h(x,0)-d\int_\Omega J(x-y)dy).$$ Let $\varphi>0$ in $\Omega$ be the corresponding principal eigenfunction with $\lambda_p(\mathfrak{L}_{\Omega,J,-q}+h(x,0)-d\int_\Omega J(x-y)dy)$. Proposition \ref{3proposition2} means $\varphi(l_1)>0,\varphi(l_2)=0$. Define the inner product $\langle\cdot,\cdot\rangle$ in $L^2(\Omega)$. Then
\begin{align*}
\langle u(t,x;&u_0),\varphi\rangle\\
> \langle&d\int_0^t\int_\Omega J(x-y)[u(s,y)-u(s,x)]dyds-\int_0^tqu_x(s,x)ds,\varphi\rangle+\langle\int_0^tf(x,u(s,x))ds,\varphi\rangle\\
\geq\langle&d\int_\Omega J(x-y)[\varphi(y)-\varphi(x)]dy+q\varphi',\int_0^tu(s,x)ds\rangle-(q\varphi\int_0^tu(s,x)ds)|_{l_1}^{l_2}\\
&+\langle(h(x,0)+\frac{\lambda_p(\mathfrak{L}_{\Omega,J,-q}+h(x,0)-d\int_\Omega J(x-y)dy)}{2})\varphi,\int_0^tu(s,x)ds\rangle\\
\geq\langle&-\frac{\lambda_p(\mathfrak{L}_{\Omega,J,-q}+h(x,0)-d\int_\Omega J(x-y)dy)}{2}\varphi,\int_0^tu(s,x)ds\rangle-(q\varphi\int_0^tu(s,x)ds)|_{l_1}^{l_2}\\
\geq-&\frac{\lambda_p(\mathfrak{L}_{\Omega,J,-q}+h(x,0))-d\int_\Omega J(x-y)dy}{2}\langle\varphi,\int_0^tu(s,x)ds\rangle.
\end{align*}
Thus, given $T>0$ large enough, for any $t>T$, there holds
\[\langle u(t,x;u_0),\varphi\rangle
\geq-\frac{\lambda_p(\mathfrak{L}_{\Omega,J,-q}+h(x,0)-d\int_\Omega J(x-y)dy)}{2}\langle\varphi,\int_0^Tu(s,x)ds\rangle.\]
Furthermore,
\[\langle \int_T^tu(s,x;u_0)ds,\varphi\rangle
\geq-\frac{\lambda_p(\mathfrak{L}_{\Omega,J,-q}+h(x,0)-d\int_\Omega J(x-y)dy)}{2}(t-T)\langle\varphi,\int_0^Tu(s,x)ds\rangle,\]
which implies that
\begin{align*}
\langle u(t,x;u_0),\varphi\rangle
&\geq-\frac{\lambda_p(\mathfrak{L}_{\Omega,J,-q}+h(x,0)-d\int_\Omega J(x-y)dy)}{2}\langle \int_T^tu(s,x;u_0)ds,\varphi\rangle\\
&\geq[\frac{\lambda_p(\mathfrak{L}_{\Omega,J,-q}+h(x,0)-d\int_\Omega J(x-y)dy)}{2}]^2(t-T)\langle\varphi,\int_0^Tu(s,x)ds\rangle.
\end{align*}
Due to $\langle\varphi,\int_0^Tu(s,x)ds\rangle>0$, we have $\lim\limits_{t\to\infty}\langle u(t,x;u_0),\varphi\rangle=+\infty$. This contradicts  with $\|u(t,x;u_0)\|_\infty<\epsilon$ for all $t>0$. Hence, $U_0$ is unstable.\\
(2) Let $(\lambda_p(\mathfrak{L}_{\Omega,J,q}+h(x,0)-d\int_\Omega J(x-y)dy),\phi)$ be the principle eigen pair of
 \[d\int_\Omega J(x-y)[\varphi(y)-\varphi(x)]dy-q\varphi'(x)+h(x,0)\varphi(x)+\lambda\varphi(x)=0,\ \ \ x\in\Omega.\]
 Define $\bar{u}=Me^{-\frac{\lambda_p(\mathfrak{L}_{\Omega,J,q}+h(x,0)-d\int_\Omega J(x-y)dy)}{2}t}\phi(x)\geq0$ in $\bar{\Omega}$  where $M>0$ satisfying $M\phi(x)\geq u_0(x)$ in $\bar{\Omega}$. Due to $h(x,u)$ decreasing in $u$,  we have
\begin{align*}
 &\bar{u}_t-\{d\int_\Omega J(x-y)[\bar{u}(t,y)-\bar{u}(t,x)]dy-q\bar{u}_x+f(x,\bar{u})\}\\
 &\geq \bar{u}_t-\{d\int_\mathbb{R} J(x-y)[\bar{u}(t,y)-\bar{u}(t,x)]dy-q\bar{u}_x+\bar{u}h(x,0)\}\\
 &\geq -\frac{\lambda_p(\mathfrak{L}_{\Omega,J,q}+h(x,0)-d\int_\Omega J(x-y)dy)}{2}\bar{u}+\lambda_p(\mathfrak{L}_{\Omega,J,q}+h(x,0)-d\int_\Omega J(x-y)dy)\bar{u}\\
 & \geq \frac{\lambda_p(\mathfrak{L}_{\Omega,J,q}+h(x,0)-d\int_\Omega J(x-y)dy)}{2}\bar{u}\geq0.
\end{align*}
The comparison principle means $0<u(t,x;u_0)\leq \bar{u}$. As $t\to\infty$, we have $\lim\limits_{t\to\infty}u(t,x;u_0)=0$, namely, $U_0$  is asymptotically stable.
\end{proof}

\begin{theorem}\label{3theorem8}
 If $\lambda_p(\mathfrak{L}_{\Omega,J,q}+h(x,0)-d\int_\Omega J(x-y)dy)\geq0$, then \eqref{3a_14} has no nontrivial solution.
\end{theorem}
\begin{proof}
{\bf{Case 1:}} $\lambda_p(\mathfrak{L}_{\Omega,J,q}+h(x,0)-d\int_\Omega J(x-y)dy)>0$. Since $\bar{u}=\max\{\max_{\bar{\Omega}} u_0(x),N\}$ is the upper  solution of \eqref{3a_14}, we deduce
\[0<u(t,x;u_0)\leq u(t,x;\bar{u})\leq\bar{u}=\max\{\max_{\bar{\Omega}} u_0(x),N\}.\]
Noting that $u(t,x;\bar{u})$ is decreasing in $t$, there exists $U(x)$ such that $\lim\limits_{t\to\infty} u(t,x;\bar{u})=U(x)$ uniformly on $\bar{\Omega}$ and $U(x)$ satisfies \eqref{3a_14}. If $U(x)\not\equiv0$, then $U(x)$ is a nontrivial solution of \eqref{3a_14}, and so, $U(x)$ is asymptotically stable by Theorem \ref{3theorem7}. This contradicts  with Lemma \ref{3lemma3}(2). 
Thus, $U(x)\equiv0$, namely,
\begin{align}\label{3a_18}
\lim\limits_{t\to\infty} u(t,x;u_0)=0.
 \end{align}
 Assume that there exists a nontrivial solution $V(x)$ of \eqref{3a_14}. Theorem \ref{3theorem7} implies that $\lim\limits_{t\to\infty} u(t,x;u_0)=V(x)$, which is a  contradiction with \eqref{3a_18}.\\
 {\bf{Case 2:}} $\lambda_p(\mathfrak{L}_{\Omega,J,q}+h(x,0)-d\int_\Omega J(x-y)dy)=0$. Assume that there exists a bounded nontrivial solution $U(x)\in(0,\max\{\max_{\bar{\Omega}} u_0(x),N\}]$ of \eqref{3a_14}. And then, $U(x)$ satisfies
 \[d\int_\Omega J(x-y)[U(y)-U(x)]dy-qU_x+Uh(x,U)=0,x\in{\Omega}\]
 Denote $b(x):=h(x,U)<h(x,0)$ in $\bar{\Omega}$. By Proposition \ref{3proposition1}(2), we have
 \[\lambda_p(\mathfrak{L}_{\Omega,J,q}+b(x)-d\int_\Omega J(x-y)dy)\geq\lambda_p(\mathfrak{L}_{\Omega,J,q}+h(x,0)-d\int_\Omega J(x-y)dy)=0.\]
Following the proof of \cite[Proposition 6.1]{Coville2021}, we deduce $\lambda_p(\mathfrak{L}_{\Omega,J,q}+b(x)-d\int_\Omega J(x-y)dy)\leq0$. Thus, it follows from Proposition \ref{3proposition1}(3) that \[\lambda_p(\mathfrak{L}_{\Omega,J,-q}+b(x)-d\int_\Omega J(x-y)dy)=\lambda_p(\mathfrak{L}_{\Omega,J,q}+b(x)-d\int_\Omega J(x-y)dy)=0.\]
By Proposition \ref{3proposition2} we can check that there is $\psi\in C^1(\Omega)\cap C(\overline{\Omega})$ satisfying $\psi>0$ in $\Omega$, $\psi(l_1)>0,\psi(l_2)=0$, such that
\begin{equation*}
d\int_{\Omega} J(x-y)[\psi(y)-\psi(x)]dy+q\psi'(x)+b(x)\psi(x)=0,\ \ \ x\in\Omega.
\end{equation*}
Define the inner product $\langle\cdot,\cdot\rangle$ in $L^2(\Omega)$. By calculation, we have
\begin{align*}
0=\langle&d\int_{\Omega} J(x-y)[\psi(y)-\psi(x)]dy+q\psi'(x)+b(x)\psi(x),U(x)\rangle\\
= \langle&d\int_{\Omega} J(x-y)[U(y)-U(x)]dy-qU'(x)+b(x)U(x),\psi(x)\rangle+(qU(x)\psi(x))|_{l_1}^{l_2}\\
= \langle&d[\int_\Omega J(x-y)[U(y)-U(x)]dy-qU'+h(x,U)U,\psi\rangle+(qU\psi)|_{l_1}^{l_2}\\
=(q&U\psi)|_{l_1}^{l_2}<0.
\end{align*}
Hence, we obtain a contradiction, and so the theorem holds.
\end{proof}
\begin{lemma}
 If $\lambda_p(\mathfrak{L}_{\Omega,J,q}+h(x,0)-d\int_\Omega J(x-y)dy)=0$,  then $U_0$  is asymptotically stable.
\end{lemma}
\begin{proof}
Let $\varphi>0$ in $\Omega$ be the principal eigenfunction with $\lambda_p(\mathfrak{L}_{\Omega,J,q}+h(x,0)-d\int_\Omega J(x-y)dy)=0$. Choose a constant $M>0$ such that $\bar{u}:=M\varphi\geq u_0(x)$. Similar to Theorem \ref{3theorem7}, $0<u(t,x;u_0)\leq u(t,x;\bar{u})\leq\bar{u}$ and there exists $U(x)$ such that $\lim\limits_{t\to\infty} u(t,x;\bar{u})=U(x)$ uniformly on $\bar{\Omega}$ and $U(x)$ satisfies \eqref{3a_14}. Combining with Theorem \ref{3theorem8}, $U(x)\equiv0$. Then $\lim\limits_{t\to\infty} u(t,x;u_0)=0$. The proof is completed.
\end{proof}
In summary, the results about the existence/nonexistence, uniqueness  of nontrivial solutions of \eqref{3a_14} and  the stability of nontrivial/trivial solution are as follows.
\begin{theorem}\label{3theorem9}
For  \eqref{3a_14}, the following results hold:
\begin{enumerate}[(1)]
 \item If $\lambda_p(\mathfrak{L}_{\Omega,J,q}+h(x,0)-d\int_\Omega J(x-y)dy)<0$, then $U_0$ is unstable. Moreover, \eqref{3a_14} admits a unique bounded nontrivial solution $U(x)\in C^1(\bar{\Omega})$ satisfying $U(x)>0$ in $(l_1,l_2]$,  which is asymptotically stable.
 \item If $\lambda_p(\mathfrak{L}_{\Omega,J,q}+h(x,0)-d\int_\Omega J(x-y)dy)\geq0$,  then $U_0$ is asymptotically stable and \eqref{3a_14} has no nontrivial solution.
\end{enumerate}
\end{theorem}

 For \eqref{3a_15}, we can also have the same analysis, which looks like the proofs of \eqref{3a_14}, so we still have the following theorem:
\begin{theorem}\label{3theorem10}
For \eqref{3a_15}, the following assertions are valid:
\begin{enumerate}[(1)]
 \item If $\lambda_p(\mathfrak{L}_{\Omega,J,q}+h(x,0))+d<0$, then $U_0$ is unstable. Moreover, \eqref{3a_15} admits a unique bounded nontrivial solution $U(x)\in C^1(\bar{\Omega})$ satisfying $U(x)>0$ in $(l_1,l_2]$, which is asymptotically stable.
 \item If  $\lambda_p(\mathfrak{L}_{\Omega,J,q}+h(x,0))+d\geq0$,  then $U_0$ is asymptotically stable and \eqref{3a_15} has no nontrivial solution.
\end{enumerate}
\end{theorem}

\section{The sharp criteria for persistence or extinction \label{3a_part4}}
\setcounter{equation}{0}
\begin{theorem}\label{3theorem12}
Assume further that $h(x,0)\geq d\int_\Omega J(x-y)dy$ on $\Omega$.  Then for \eqref{3a_11}, there is $q_*>0$ such that the persistence occurs if $0< q<q_*$ and the extinction occurs if $q\geq q_*$.
\end{theorem}
\begin{proof}
Without loss of generality, we assume $q>0$. Denote the principal eigen pair $(\lambda_p(\mathfrak{L}_{\Omega,J,q}+h(x,0)-d\int_\Omega J(x-y)dy),\varphi)$ satisfying $\varphi>0$ in $\Omega$, $\varphi(l_1)=0,\varphi(l_2)>0$ and
\[d\int_\Omega J(x-y)[\varphi(y)-\varphi(x)]dy-q\varphi'(x)+h(x,0)\varphi(x)+\lambda_p(\mathfrak{L}_{\Omega,J,q}+h(x,0)-d\int_\Omega J(x-y)dy)\varphi(x)=0,\ \ \ x\in\Omega.\]
Integrating both sides of the above equations in $\Omega$, we deduce
\begin{align*}
-&\lambda_p(\mathfrak{L}_{\Omega,J,q}+h(x,0)-d\int_\Omega J(x-y)dy)\int_\Omega \varphi(x) dx\\
&=d\int_\Omega \int_\Omega J(x-y)[\varphi(y)-\varphi(x)]dydx-q\varphi|_{l_1}^{l_2}+\int_\Omega h(x,0)\varphi(x)dx\\
&=\int_\Omega h(x,0)\varphi(x)dx-q\varphi(l_2).
\end{align*}
Since $\varphi(x)$ and $h(x,0)$ are bounded and positive, $\lambda_p(\mathfrak{L}_{\Omega,J,q}+h(x,0)-d\int_\Omega J(x-y)dy)<0$ for $q\to0$ and $\lambda_p(\mathfrak{L}_{\Omega,J,q}+h(x,0)-d\int_\Omega J(x-y)dy)>0$ for $q\gg1$. Employing Proposition \ref{3proposition1}(4) and Proposition \ref{3proposition3}, since $h(x,0)-d\int_\Omega J(x-y)dy\geq0$,  we can find $q_*>0$ such that $\lambda_p(\mathfrak{L}_{\Omega,J,q}+h(x,0)-d\int_\Omega J(x-y)dy)<0$ if $0< q<q_*$ and $\lambda_p(\mathfrak{L}_{\Omega,J,q}+h(x,0)-d\int_\Omega J(x-y)dy)\geq0$ if $q\geq q_*$. Combining with Theorem \ref{3theorem9}, the theorem holds.
\end{proof}

Define
\[\bar{h}=\sup_{\bar{\Omega}}h(x,0)>0, \underline{h}=\inf_{\bar{\Omega}}h(x,0);\]
\[q^*=\inf_{\mu>0}\frac{d\int_{\mathbb{R}}J(z)e^{-\mu z}dz-d+\bar{h}}{\mu}.\]
By the assumption {\bf{(J0)}} and the definition of $h(x,0)$, we can obtain from \cite{Shen2010} that $q^*\in(0,+\infty)$  and there is $\mu^*>0$ such that $$d\int_{\mathbb{R}}J(z)e^{-\mu^* z}dz-d+\bar{h}=\mu^*q^*.$$

\begin{theorem}\label{3theorem11}
Assume further that $\inf_{\bar{\Omega}}h(x,0)> d$.  Then for \eqref{3a_1}, the persistence  occurs if $0< q<q^{**}$ and the extinction occurs if $q\geq q^{**}$.
\end{theorem}
\begin{proof}
Firstly, set $\bar{\lambda}=d\int_{\mathbb{R}}J(z)e^{-\mu^* z}dz-d+\bar{h}-q\mu^*$ and $\phi_1(x)=e^{\mu^*x}$ in $\mathbb{R}$. Then $(\bar{\lambda},\phi_1)$ satisfies
\begin{align*}
 d\int_\mathbb{R} J(x-y)\phi_1(y)dy-q\phi'_1+\bar{h}\phi_1=(\bar{\lambda}+d)\phi_1, \ \ \mbox{on} \  \mathbb{R}.
\end{align*}
By the definition of \eqref{3a_13}, we have  $\lambda_p(\mathfrak{L}_{\mathbb{R},J,q}+\bar{h})\geq-(\bar{\lambda}+d)$. In terms of Proposition \ref{3proposition1}(1)(2), there holds
\[\lambda_p(\mathfrak{L}_{\mathbb{R},J,q}+\bar{h})\leq \lambda_p(\mathfrak{L}_{\mathbb{R},J,q}+h(x,0))\leq \lambda_p(\mathfrak{L}_{\Omega,J,q}+h(x,0)).\]
Thus,
\[\lambda_p(\mathfrak{L}_{\Omega,J,q}+h(x,0))+d \geq -\bar{\lambda}=(q-q^*)\mu^*.\]
When $q \geq q^*$, we deduce $\lambda_p(\mathfrak{L}_{\Omega,J,q}+h(x,0))+d\geq0$.

Next, set $\underline{\lambda}=d\int_{\mathbb{R}}J(z)e^{-\mu^* z}dz-d+\underline{h}-q\mu^*$ and $\phi_2(x)=e^{\mu^*x}$ in $\Omega$. Then $(\underline{\lambda},\phi_2)$ satisfies
\begin{align*}
 d\int_\Omega J(x-y)\phi_2(y)dy-q\phi'_2+\underline{h}\phi_2&=(\underline{\lambda}+d)\phi_2-d\int_{\mathbb{R}\backslash \Omega} J(x-y)e^{\mu^*y }dy\\
 &\geq (\underline{\lambda}+d-d\int_{\mathbb{R}}J(z)e^{-\mu^* z}dz)\phi_2, \ \ \ \ \ \ \mbox{on} \  \Omega.
\end{align*}
According to Proposition \ref{3proposition2}, $\lambda_p(\mathfrak{L}_{\Omega,J,q}+\underline{h})\leq-(\underline{\lambda}+d-d\int_{\mathbb{R}}J(z)e^{-\mu^* z}dz)$. Since $\lambda_p(\mathfrak{L}_{\Omega,J,q}+h(x,0))\leq\lambda_p(\mathfrak{L}_{\Omega,J,q}+\underline{h})$,
\[\lambda_p(\mathfrak{L}_{\Omega,J,q}+h(x,0))+d\leq -(\underline{\lambda}-d\int_{\mathbb{R}}J(z)e^{-\mu^* z}dz)=q\mu^*+d-\underline{h}.\]
When $0< q < \frac{\underline{h}-d}{\mu^*}<q^*$, we have $\lambda_p(\mathfrak{L}_{\Omega,J,q}+h(x,0))+d<0$. Proposition \ref{3proposition1}(4) and Proposition \ref{3proposition3} imply that there is $q^{**}>0$ such that $\lambda_p(\mathfrak{L}_{\Omega,J,q}+h(x,0))+d<0$ if $0< q<q^{**}$ and $\lambda_p(\mathfrak{L}_{\Omega,J,q}+h(x,0))+d\geq0$ if $q\geq q^{**}$.

By Theorem \ref{3theorem10}, the solution $u(t,x;u_0)$ of  \eqref{3a_1} satisfies $\lim\limits_{t\to\infty} u(t,x;u_0)=0$ for $q \geq q^{**}$ and $\lim\limits_{t\to\infty} u(t,x;u_0)=U(x)>0$ for $0<q < q^{**}$.  The proof is completed.
\end{proof}

\begin{remark}
Noting that $$\lambda_p(\mathfrak{L}_{\Omega,J,q}+h(x,0))+d=\lambda_p(\mathfrak{L}_{\Omega,J,q}+h(x,0)-d)\geq \lambda_p(\mathfrak{L}_{\Omega,J,q}+h(x,0)-d\int_\Omega J(x-y)dy),$$
we have $q^{**}\leq q^{*}$.
\end{remark}
\section{The limiting behaviors of solutions with respect to advection \label{3a_part5}}
\setcounter{equation}{0}
In this section, we analyze the limiting behaviors of positive solution to \eqref{3a_11}(or \eqref{3a_1}) when $q\to 0^+/+\infty$. Meanwhile, if the nontrivial stationary solution of \eqref{3a_11}(or \eqref{3a_1}) exists, the limiting behaviors also can be studied. Firstly, we need to show that the positive solution of \eqref{3a_11}(or \eqref{3a_1}) and the nontrivial stationary solution of \eqref{3a_11}(or \eqref{3a_1}) are continuous with respect to $q$.

\begin{lemma}\label{3lemma4}
 Let $u_q(t,x)$ be the unique positive solution of \eqref{3a_11}(or \eqref{3a_1}) for $q>0$. Then $u_q(t,x)$  is continuous with respect to $q>0$ for any given $t>0$ and  $x\in\bar{\Omega}$.
\end{lemma}
\begin{proof}
We only prove the conclusion for \eqref{3a_11}. 
\begin{equation*}
\begin{cases}
        w_t(t,x)=d\int_\Omega J(x-y)[w(t,y)-w(t,x)]dy-q_1w_x+(q_2-q_1)(u_{q_2})_x\\
        \quad\quad\quad\quad\quad\quad\quad  +u_{q_1}h(x,u_{q_1})-u_{q_2}h(x,u_{q_2}),&t>0,x\in{\Omega},\\
        w(t,l_1)=0,&t>0,\\
       w(0,x)=0,&x\in{\Omega}.
        \end{cases}
 \end{equation*}
Noticing that
\begin{align*}
u_{q_1}h(x,u_{q_1})-u_{q_2}h(x,u_{q_2})&=u_{q_1}[h(x,u_{q_1})-h(x,u_{q_2})]+(u_{q_1}-u_{q_2})h(x,u_{q_2})\\
&=u_{q_1}\left(\int_0^1\frac{\partial h(x,\xi)}{\partial\xi}\mid_{\xi=u_{q_2}+s(u_{q_1}-u_{q_2})}ds\right)w+h(x,u_{q_2})w\\
&=\left(u_{q_1}\int_0^1\frac{\partial h(x,\xi)}{\partial\xi}\mid_{\xi=u_{q_2}+s(u_{q_1}-u_{q_2})}ds+h(x,u_{q_2})\right)w,
\end{align*}
it follows from $h(x,u)\in C^{1}(\bar{\Omega}\times\mathbb{R}^+)$ that there exists $M>0$ such that
\[|u_{q_1}h(x,u_{q_1})-u_{q_2}h(x,u_{q_2})|\leq Mw.\]
Since $(u_{q_1})_x$ is uniformly bounded on $Q_\infty$ by \eqref{3a_a2}, we define $v(t,x)=(q_1-q_2)e^{At}>0$, where $A$ is a positive constant independent of $q_1,q_2$ and satisfying $A>\|(u_{q_2})_x\|_{L^\infty}+M$. Then
\begin{align*}
 v_t(t,x)-&\left\{d\int_\Omega J(x-y)[v(t,y)-v(t,x)]dy-q_1v_x+(q_2-q_1)(u_{q_2})_x+u_{q_1}h(x,u_{q_1})-u_{q_2}h(x,u_{q_2})\right\}\\
 &\geq Av-\|(u_{q_2})_x\|_{L^\infty}e^{At}(q_1-q_2)-Mv\\
 &\geq [A-\|(u_{q_2})_x\|_{L^\infty}-M]v\geq0.
\end{align*}
By  comparison principle, the initial value $v(0,x)=q_1-q_2>0$  deduces
\[|w(t,x)|=|u_{q_1}(t,x)-u_{q_1}(t,x)|\leq v(t,x)=(q_1-q_2)e^{At},\]
namely, $u_q(t,x)$  is continuous with respect to $q>0$  for any given $t>0$ and  $x\in\bar{\Omega}$. The proof is completed.
\end{proof}

\begin{lemma}
 Assume that $I\subset(0,\infty)$  and let $u_q(x)$ be the unique nontrivial stationary solution of \eqref{3a_11} (or \eqref{3a_1}) for $q\in I$. Then $u_q(x)$  is continuous with respect to $q\in I$.
\end{lemma}
\begin{proof}
The proof is inspired by \cite{Ouyang1992,Sun2021}. We only prove it for \eqref{3a_11}. It suffices to show that  $u_q(x)$  is differentiable with respect to $q\in I$.  Obviously, $u_q(x)>0$ on ${\Omega}$. For any $(q,\varphi)\in I\times C^1(\bar{\Omega})$, define the operator
\[\mathcal{L}(q,\varphi)=d\int_\Omega J(x-y)[\varphi(y)-\varphi(x)]dy-q\varphi'+\varphi h(x,\varphi).\]
 The corresponding linear operator with $\varphi$ at $\varphi=u_q$ is
\[\mathcal{L}_\varphi(q,u_q)v=d\int_\Omega J(x-y)[v(y)-v(x)]dy-qv'+[h(x,u_q)+u_qh_u(x,u_q)]v.\]
Denote $\lambda$ as the principal eigenvalue of operator $\mathcal{L}_\varphi(q,u_q)$. By Proposition \ref{3proposition1}(3), $\lambda$ is also the principal eigenvalue of operator $\mathcal{L}_\varphi(-q,u_q)$ and can be written as
\begin{align*}
\lambda:=\inf\{\lambda\in\mathbb{R}|\exists \varphi\in &C^1(\Omega)\cap C(\bar{\Omega}) \ \mbox{satisfying} \ \varphi>0 \ \mbox{in} \ \Omega,\varphi(l_2)=0,\varphi(l_1)>0, \\  &\mbox{and} \
\mathfrak{L}_{\Omega,J,q}[\varphi(x)]+a(x)\varphi(x)+\lambda\varphi(x)\geq0\},
\end{align*}
by Proposition \ref{3proposition2}.  That is to say, there exists a principal eigenfunction $\phi$ satisfying $\phi(x)>0$ in $\Omega$, $\phi(l_2)=0,\phi(l_1)>0$ such that
\begin{align}\label{3a_20}
d\int_\Omega J(x-y)[\phi(y)-\phi(x)]dy+q\phi'+[h(x,u_q)+u_qh_u(x,u_q)]\phi+\lambda\phi\geq0 \ \mbox{in} \ \Omega.
\end{align}
In terms of $\mathcal{L}(q,u_q)=0$, namely,
\begin{align}\label{3a_21}
d\int_\Omega J(x-y)[u_q(y)-u_q(x)]dy-qu_q'+u_q h(x,u_q)=0,
\end{align}
we multiply \eqref{3a_20} by $u_q$ and \eqref{3a_21} by $\phi$, and then, integrate both sides of the above equations in $\Omega$ and subtract. Together with $h(x,u)$ decreasing in $u$, we obtain
\begin{align*}
-\lambda\int_\Omega\phi u_qdx\leq q(\phi u_q)|_{l_1}^{l_2}+\int_\Omega\phi u_q^2h_u(x,u_q)dx<0.
\end{align*}
That is to say, $\lambda>0$ and operator $\mathcal{L}_\varphi(q,u_q)$ is invertible. The implicit function theorem means that the conclusion holds.
\end{proof}
\begin{theorem}\label{3theorem13}
Let $u_q(x)$ be the unique nontrivial stationary solution of \eqref{3a_11}  for $q\in (0,q^*)$. Then
$$\lim\limits_{q\to0^+} u_q(x)=\gamma(x) \ \mbox{uniformly in} \ \bar{\Omega},\ \lim\limits_{q\to (q^*)^-} u_q(x)=0 \ \mbox{uniformly in} \ \bar{\Omega},$$
     where $\gamma(x)$ is the unique positive solution of
     \[d\int_\Omega J(x-y)[\gamma(y)-\gamma(x)]dy+\gamma h(x,\gamma)=0,x\in{\Omega}.\]
\end{theorem}
\begin{proof}
By Section 3, we know that for $q\in (0,q^*)$, the function $u_q(x)$ satisfies $u_q(l_1)=0$, $u_q(x)>0$ in $(l_1,l_2]$, and
\[0\leq u_{q}(x) \leq \max\{\max_{\bar{\Omega}}u_0(x), N\}:=M,\]
which means $u'_q(l_1)>0$ and $u_q(x)$ is uniformly bounded.

We first choose a small enough $\rho>0$, and consider the case $q\geq \rho$. Let $z(x):=(M_1-u_{q}'(x))e^{-N_1(x-l_1)}$ for some constant $M_1\geq u'_q(l_1)$ and $N_1>0$. Obviously, $z(l_1)=M_1- u'_q(l_1)\geq0$. We shall obtain $z(x)\geq0$. Otherwise,  there is $x_0\in (l_1,l_2]$ such that $z(x_0)=\inf_{[l_1,l_2]}z(x)<0$.
And so, $z'(x_0)\leq0$. Moreover, a simple computation implies that
\begin{align*}
qz'(x)&=-qN_1z(x)-(qu_{q}'(x))'e^{-N_1(x-l_1)}\\
&=-qN_1z(x)+d\int_\Omega J(x-y)[z(y)-z(x)]dy+k_1(x,u_q)z(x)+k_2(x,u_q)e^{-N_1(x-l_1)}\\
&=d\int_\Omega J(x-y)[z(y)-z(x)]dy+(k_1(x,u_q)-qN_1)z(x)+k_2(x,u_q)e^{-N_1(x-l_1)},
\end{align*}
where
\begin{align*}
k_1(x,u_q)&=du_q(l_2)J(x-l_2)+d\int_\Omega  J_x(x-y)dyu_q-u_qh_x(x,u_q)-M_1h(x,u_q)-M_1u_qh_u(x,u_q);\\
k_2(x,u_q)&=h(x,u_q)+u_qh_u(x,u_q).
\end{align*}
By $0<u_q<M$ and $h(x,u)\in C^1([l_1,l_2]\times \mathbb{R})$, we have that $k_i(x,u_q)(i=1,2)$ is uniformly bounded. That is, there there is $M_2>0$ such that $|k(x,u_q)|\leq M_2$. Then
\begin{align*}
0&\geq qz'(x_0)\\
&=d\int_\Omega J(x_0-y)[z(y)-z(x_0)]dy+(k_1(x_0,u_q(x_0))-qN_1)z(x_0)+k_2(x_0,u_q(x_0))e^{-N_1(x_0-l_1)}\\
&\geq (M_2-\rho N_1)z(x_0)-M_2e^{-N_1(x_0-l_1)}>0,
\end{align*}
if we choose $N_1>0$ large enough such that
\[(M_2-\rho N_1)z(x_0)-M_2e^{-N_1(x_0-l_1)}>0.\]
This is a contradiction, and then $u_{q}'(x)\leq M_1$ for $q\geq\rho$. Similarly, we also have $u_{q}'(x)\geq -M_1$ for $q\geq\rho$. And so, $u_q'(x)$ is uniformly bounded for $q\geq\rho$.  It follows that $u_q(x)$ is equi-continuous for $q\geq\rho$.

For the case $q\in[0,\rho]$, since $u_q(x)$ is continuous in $q\in[0,\rho]$, it deduces that, for any $\epsilon>0$ and any given $x\in\bar{\Omega}$, there exists $\delta_1(x)>0$ independent of $q$ such that for any $q\in[0,\rho]$ satisfying $|q-\rho|<\delta_1(x)$,  we have \[|u_{q}(x)-u_{\rho}(x)|<\epsilon.\]
By $x\in\bar{\Omega}$, we can choose $\delta_3=\min\{\min_{x\in\bar{\Omega}}\delta_1(x),\rho\}>0$ such that for any $q\in[0,\rho]$, we have $|q-\rho|<\delta_3$ and \[|u_{q}(x)-u_{\rho}(x)|<\epsilon.\]
Meanwhile, since $u_\rho(x)$ is continuous in $x\in\bar{\Omega}$, we know that, for any $\epsilon>0$, there exists $\delta_2(\rho)>0$ independent of $q,x$ such that for any $x,y\in\bar{\Omega}$ satisfying $|x-y|<\delta_2(\rho)$, we have \[|u_{\rho}(x)-u_{\rho}(y)|<\epsilon.\]
Then, for any $\epsilon>0$, there exists $\delta=\min\{\min_{x\in\bar{\Omega}}\delta_1(x),\rho,\delta_2(\rho)\}>0$ independent of $q,x$ such that for any $x,y\in\bar{\Omega}$ satisfying $|x-y|<\delta$ and  for any $q\in[0,\rho]$, we have \[|u_{q}(x)-u_{q}(y)|\leq|u_{q}(x)-u_{\rho}(x)|+|u_{\rho}(x)-u_{\rho}(y)|+|u_{\rho}(y)-u_{q}(y)|<3\epsilon.\]
It follows that $u_q(x)$ is equi-continuous in $q\in[0,\rho]$.

Combining with the fact that $u_q(x)$ is uniformly bounded, by Ascoli-Arzela theorem, there is a nonnegative function $\theta(x)\in C(\bar{\Omega})$ such that
$$\lim\limits_{q\to0^+} u_q(x)=\theta(x) \ \mbox{uniformly in} \ \bar{\Omega}.$$
Thus, as $q\to0^+$, we have
 \[d\int_\Omega J(x-y)[\theta(y)-\theta(x)]dy+\theta h(x,\theta)=0.\]
According to \cite{Berestycki2016}, the above equation has a unique positive solution for all $d>0$.

On the other hand, we  can also find a nonnegative function $\eta(x)\in C(\bar{\Omega})$ such that
$$ \lim\limits_{q\to (q^{*})^-} u_q(x)=\eta(x) \ \mbox{uniformly in} \ \bar{\Omega}$$
by Ascoli-Arzela theorem. Furthermore, $\eta(x)$ satisfies
\[d\int_\Omega J(x-y)[\eta(y)-\eta(x)]dy-q^*\eta'+\eta h(x,\eta)=0.\]
However, Theorem \ref{3theorem12} implies that the above equation has no nontrivial solution when $q=q^*$, namely, $\eta(x)\equiv0$ in ${\Omega}$. The proof is completed.
\end{proof}
\begin{theorem}
Let $u_q(t,x)$ be the unique positive solution of \eqref{3a_11} for $q>0$. For given $T>0$, we have
$$\lim\limits_{q\to0^+} u_q(t,x)=v^*(t,x) \ \mbox{uniformly in} \ [0,T]\times\bar{\Omega},\ \lim\limits_{q\to +\infty} u_q(t,x)=0 \ \mbox{uniformly in} \ [0,T]\times\bar{\Omega},$$
     where $v^*(t,x)$ is the unique positive solution of
    \begin{equation}\label{3a_25}
\begin{cases}
        v^*_t(t,x)=d\int_\Omega J(x-y)[v^*(t,y)-v^*(t,x)]dy+v^*h(x,v^*),&0<t<T,x\in{\Omega},\\
       v^*(0,x)=u_0(x),&x\in{\Omega}.
        \end{cases}
 \end{equation}
\end{theorem}
\begin{proof}
Similar to the proof of Theorem \ref{3theorem13}, by the simple compact argument (Ascoli-Arzela theorem), there exist nonnegative functions $u^*(t,x), w^*(t,x)\in C^1([0,T]\times\bar{\Omega})$ such that
$$\lim\limits_{q\to0^+} u_q(t,x)=u^*(t,x) \ \mbox{uniformly in} \ [0,T]\times\bar{\Omega},$$
$$ \lim\limits_{q\to +\infty} u_q(t,x)=w^*(t,x) \ \mbox{uniformly in} \ [0,T]\times\bar{\Omega}.$$
Clearly,
\[u_q(t,x)-u_0(x)=\int_0^t[d\int_\Omega J(x-y)[u_q(s,y)-u_q(s,x)]dy-q(u_q)_x(s,x)+u_q(s,x)h(x,u_q(s,x))]ds.\]
Since $u_q'$ is bounded in $[0,T]\times\bar{\Omega}$, as $q\to0^+$, it follows that
\[u^*(t,x)-u_0(x)=\int_0^t[d\int_\Omega J(x-y)[u^*(s,y)-u^*(s,x)]dy+u^*(s,x)h(x,u^*(s,x))]ds.\]
Consequently, $u^*(t,x)$ satisfies  \eqref{3a_25}. Noting that \eqref{3a_25} admits a unique positive solution by \cite{Sun2021}, there holds
$$\lim\limits_{q\to0^+} u_q(t,x)=u^*(t,x) \ \mbox{uniformly in} \ [0,T]\times\bar{\Omega}.$$

Now we consider the case for $q\to +\infty$. Taking a smooth nonnegative function $\zeta(t)$ satisfying $\zeta(0)=\zeta(T)=0$. Then, we  multiply the first equation of \eqref{3a_11} by
$\zeta(t)$ and integrate over $[0,T]$, namely,
\begin{align*}
\int_0^T \zeta(t)(u_q)_x(t,x)dt=&\frac{d}{q}\int_0^T\int_\Omega J(x-y)[u_q(t,y)-u_q(t,x)]dy\zeta(t)dt+\frac{1}{q}\int_0^Tu_q(t,x)\zeta'(t)dt\\
&+\frac{1}{q}\int_0^T[u_q(t,x)h(x,u_q(t,x))]\zeta(t)dt.
\end{align*}
Together with the boundedness of $u_q$ in $\overline{Q_T}$, as $q\to +\infty$, we deduce
\[\int_0^T \zeta(t)(w^*)_x(t,x)dt=0.\]
The arbitrary of $\zeta(t)$ implies that $(w^*)_x(t,x)\equiv0$ for all $t\in[0,T]$, and it follows that
\[w^*(t,x)=w^*(t) \ \mbox{for all}\ [0,T]\times\bar{\Omega}.\]
Moreover, for $t\in[0,T]$, we observe
\begin{align*}
\int_0^t u_q(s,x)ds=&\int_0^t\int_{l_1}^x (u_q)_x(s,z)dzds\\=&\frac{d}{q}\int_0^t\int_{l_1}^x\int_\Omega J(z-y)[u_q(s,y)-u_q(s,z)]dydzds\\&+\frac{1}{q}\int_0^t\int_{l_1}^x[u_q(s,z)h(z,u_q(s,z))]dzds-\frac{1}{q}\int_{l_1}^x [u_q(t,z)-u_0(z)]dz.
\end{align*}
As $q\to +\infty$, we know
\[\int_0^t w^*(s)ds=0 \ \mbox{for all}\ t\in[0,T],\]
namely, $w^*(t)\equiv0$ on $[0,T]$. This has proved
 $$\lim\limits_{q\to +\infty} u_q(t,x)=0 \ \mbox{uniformly in} \ [0,T]\times\bar{\Omega}.$$
The proof is completed.
\end{proof}
\begin{remark}
Note that $\gamma(x)$ is the unique nontrivial stationary solution of \eqref{3a_25}.
\end{remark}

Similarly, for \eqref{3a_1}, we have
\begin{theorem}\label{3a_3atheorem1}
Let $u_q(x)$ be the unique nontrivial stationary solution of \eqref{3a_1}  for $q\in (0,q^{**})$. Then
$$\lim\limits_{q\to0^+} u_q(x)=\theta(x) \ \mbox{uniformly in} \ \bar{\Omega},\ \lim\limits_{q\to (q^{**})^-} u_q(x)=0 \ \mbox{uniformly in} \ \bar{\Omega},$$
     where $\theta(x)$ is the unique positive solution of
     \[d\int_\Omega J(x-y)\theta(y)dy-d\theta+\theta h(x,\theta)=0, x\in\Omega.\]
\end{theorem}
\begin{theorem}\label{3a_3atheorem2}
Let $u_q(t,x)$ be the unique positive solution of \eqref{3a_1} for $q>0$. For given $T>0$, we have
$$\lim\limits_{q\to0^+} u_q(t,x)=u^*(t,x) \ \mbox{uniformly in} \ [0,T]\times\bar{\Omega},\ \lim\limits_{q\to +\infty} u_q(t,x)=0 \ \mbox{uniformly in} \ [0,T]\times\bar{\Omega},$$
     where $u^*(t,x)$ is the unique positive solution of
    \begin{equation}\label{3a_22}
\begin{cases}
        u^*_t(t,x)=d[\int_\Omega J(x-y)u^*(t,y)dy-u^*]+u^*h(x,u^*),&0<t<T,x\in{\Omega},\\
       u^*(0,x)=u_0(x),&x\in{\Omega}.
        \end{cases}
 \end{equation}
\end{theorem}
\begin{remark}
Note that $\theta(x)$ is the unique nontrivial stationary solution of \eqref{3a_22}.
\end{remark}

\section{Numerical simulations\label{3a_part6}}
\setcounter{equation}{0}
In this section,  we apply some numerical simulations to support theoretical results in previous sections and observe  how advection  rate  $q$ affects the longtime behaviors of solutions.  Take equation  \eqref{3a_1} with $q>0$ for example. Then we take  $\Omega=(0,5)$ and $d=0.26$. Meanwhile, define the initial function $u_0(x)=sin \frac{\pi x}{5}$, kernel function
$J(x)=\frac{1}{\sqrt{2\pi}}e^{-\frac{x^2}{2}}$, and
\[f(t,x,u)=(\frac{5}{2}-\frac{x^2}{16}-u)u.\]

{\bf{Example 1.}}   A numerical sample for case $q=0.5$ is presented in Fig. \ref{3a_figure1}. We observe  that  species $u$ will ultimately survive and its species density reaches a steady state as time tends to infinite in domain $\Omega=(0,5)$. And from  Fig. \ref{3a_figure1}, we find that the advection rate has an impact on the steady state of species density.  A numerical sample for case $q=1$ is shown in Fig. \ref{3a_figure2}. It follows that species $u$ will extinct in domain $\Omega=(0,5)$, which implies that its species density will tend to $0$ as time tends to infinite. Then Figs. \ref{3a_figure1} and \ref{3a_figure2} depict the temporal dynamics of species $u$ for system  \eqref{3a_1}.
 \begin{figure}[h]
  \centering
  \includegraphics[height=60mm,width=180mm]{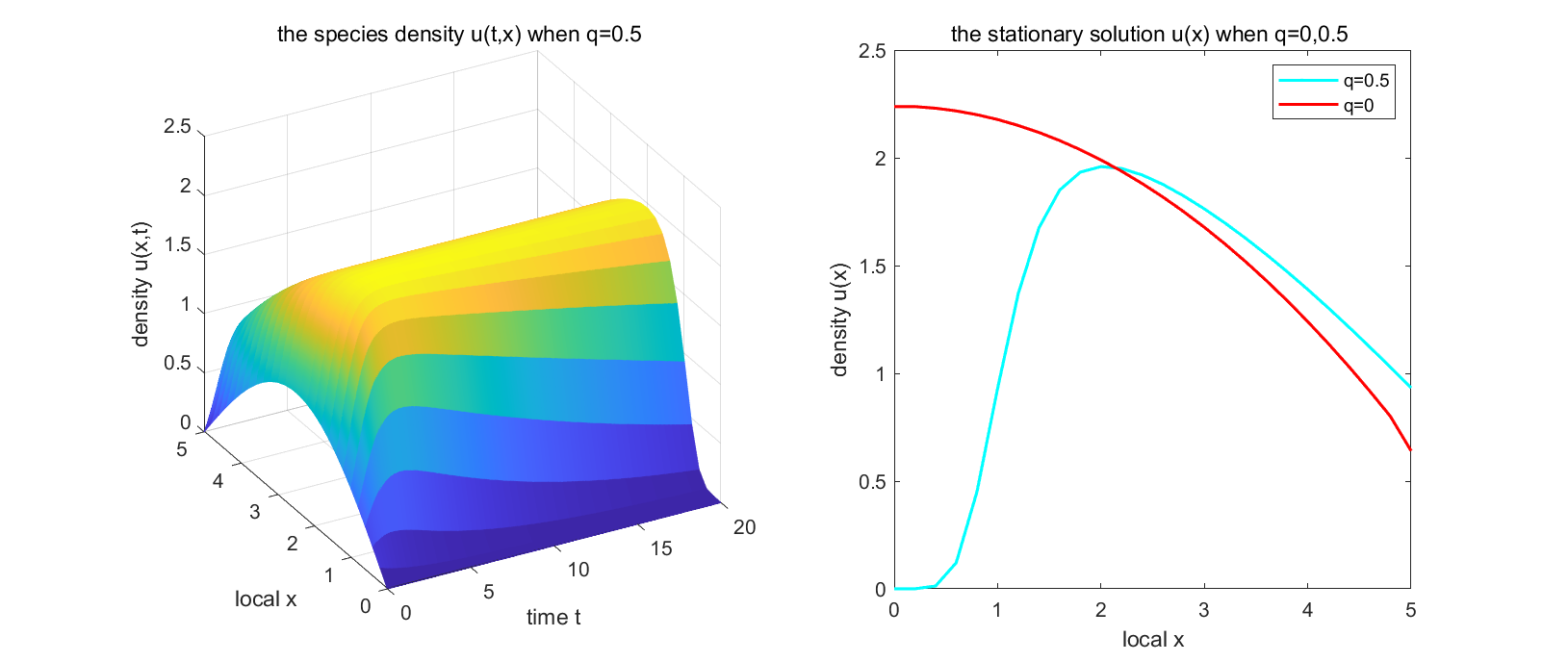}\\
  \caption{Persistence of $u$ for $q=0.5$}\label{3a_figure1}
\end{figure}

 \begin{figure}[h]
  \centering
  \includegraphics[height=60mm,width=180mm]{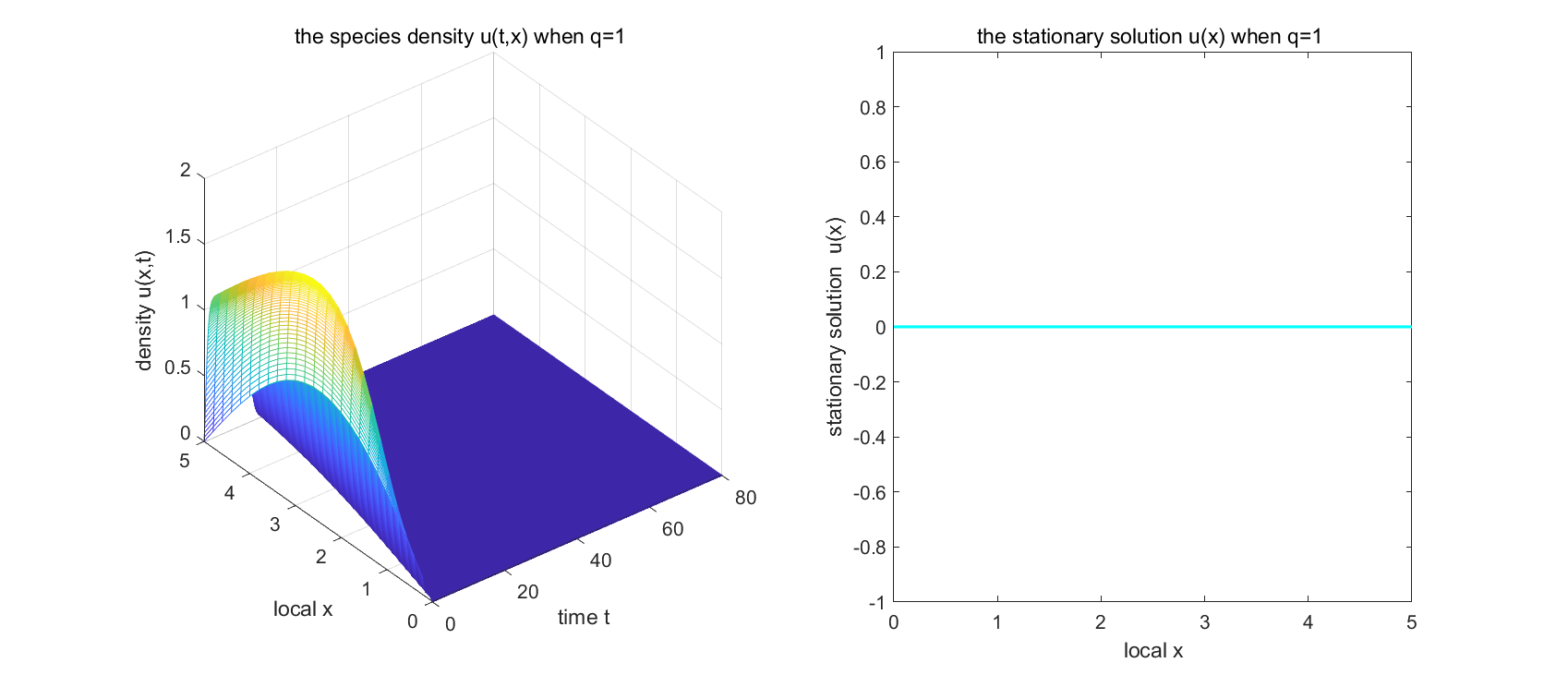}\\
  \caption{Extinction of $u$ for $q=1$}\label{3a_figure2}
\end{figure}

\begin{figure}[h]
  \centering
  \includegraphics[height=60mm,width=80mm]{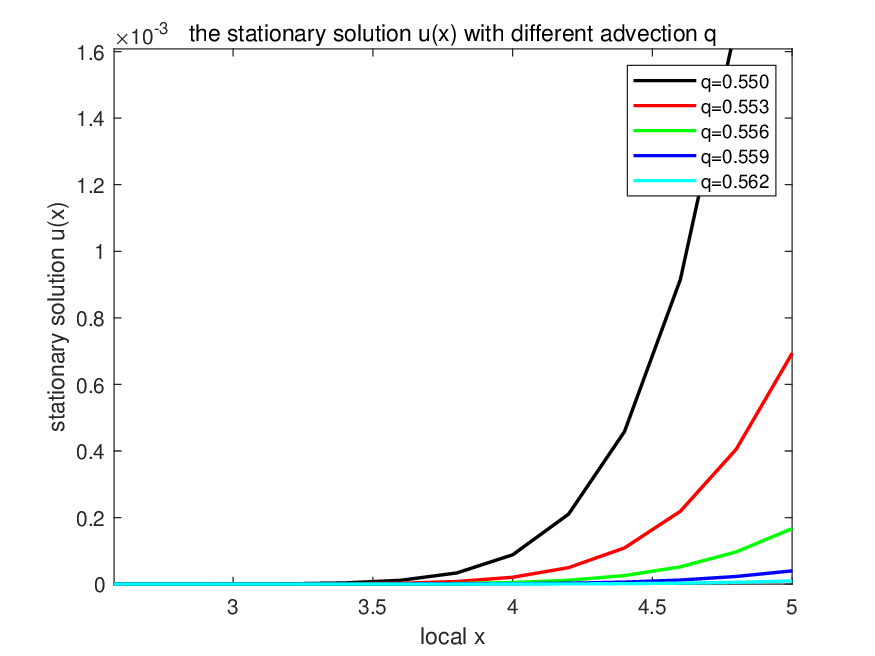}\\
  \caption{The estimation for threshold value $q^{**}$}\label{3a_figure8}
\end{figure}
Notice that $h(x,0)=\frac{5}{2}-\frac{x^2}{16}>d$. It follows from Fig. \ref{3a_figure8} that  the dependence of $q$ on stationary solution is obtained if it exists and we get the threshold value $q^{**}\in(0.559,0.562)$ for the existence of stationary solution.
This verifies the sharp criteria for persistence or extinction(see Theorem \ref{3theorem11}).

{\bf{Example 2.}} Choosing $q=0,0.1,0.2,0.3$, by Theorem \ref{3theorem11}, the species will survive and  has a steady-state.  Figs. \ref{3a_figure3} and \ref{3a_figure4} describe the profiles of $u$ with different $q$ at time $t=1,4$ and the profiles of $u$ with different $q$ at location $x=1,4$. We find that when the advection rate has a tendency to $0$, the species density will close to the species density for $q=0$. 
Meanwhile,  Fig. \ref{3a_figure5} shows that the stationary solution will close to the stationary solution for $q=0$ when the advection rate $q$ tends to $0$. These phenomena explain our results for Theorems \ref{3a_3atheorem1} and \ref{3a_3atheorem2} with $q\to0^+$.
\begin{figure}[h]
  \centering
  \includegraphics[height=60mm,width=180mm]{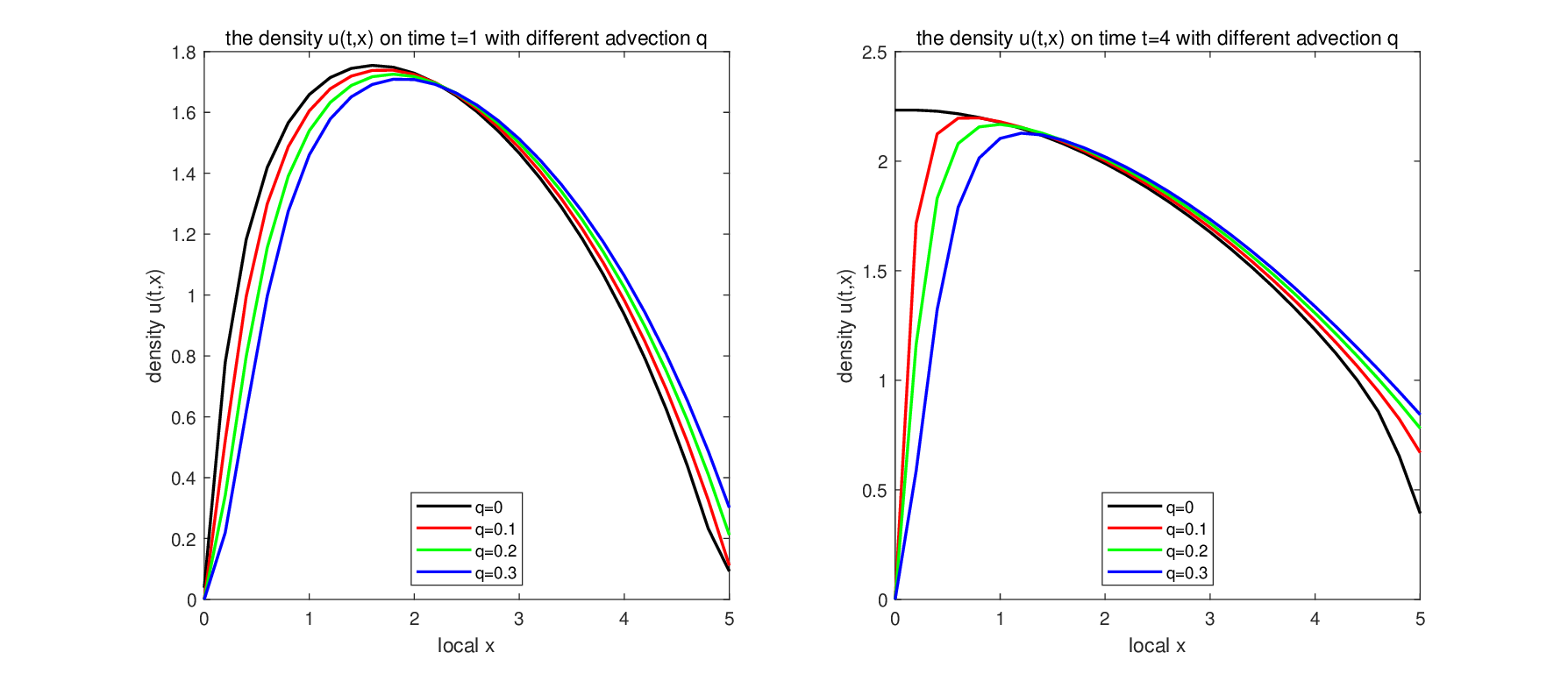}\\
  \caption{The species density of $u$  with different $q$ at time $t=1,4$ }\label{3a_figure3}
\end{figure}
\begin{figure}[h]
  \centering
  \includegraphics[height=60mm,width=180mm]{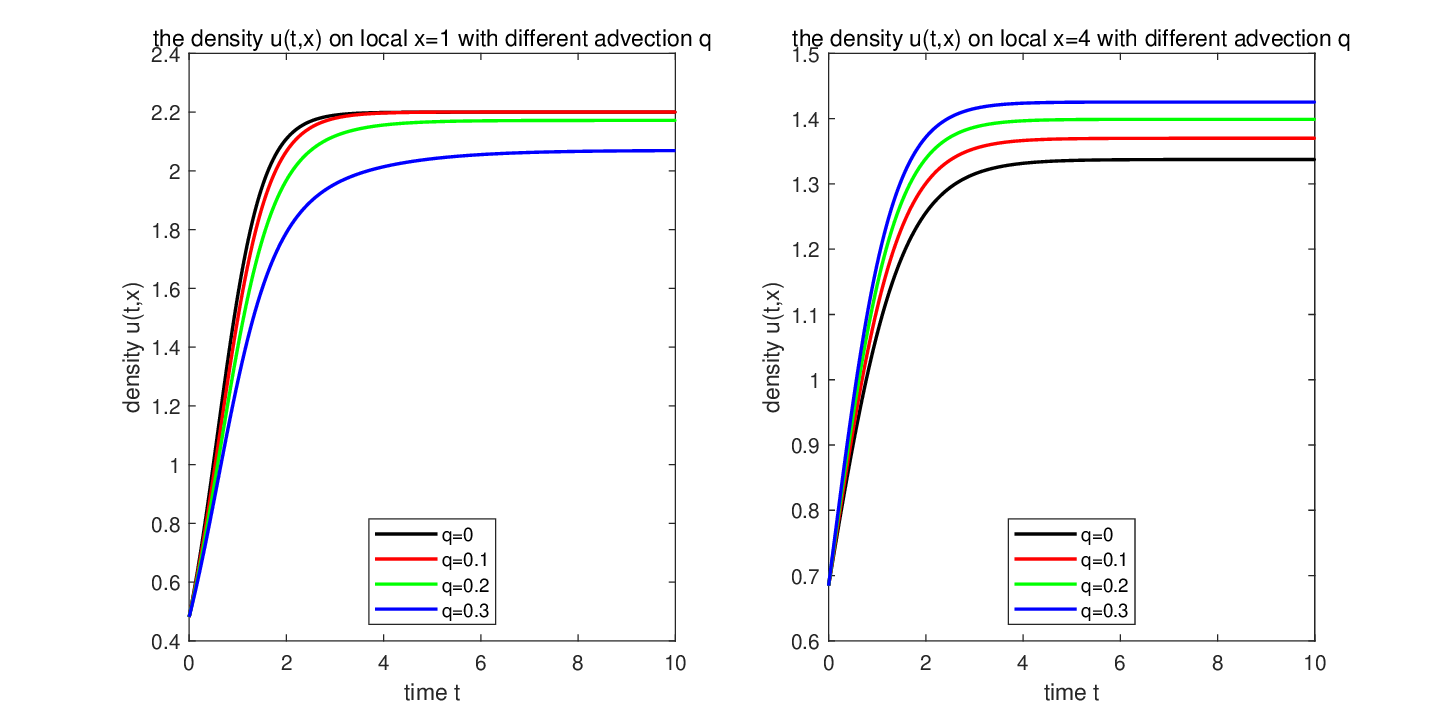}\\
  \caption{The species density of $u$  with different $q$ at location $x=1,4$}\label{3a_figure4}
\end{figure}
\begin{figure}[h]
  \centering
  \includegraphics[height=60mm,width=80mm]{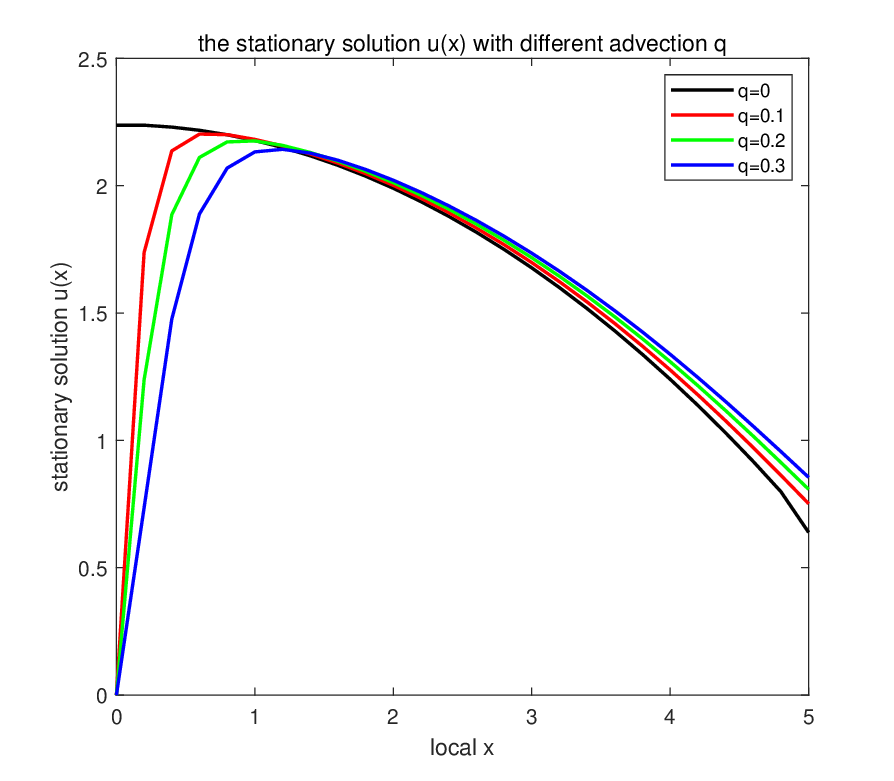}\\
  \caption{The stationary solution $u(x)$  with different $q$}\label{3a_figure5}
\end{figure}

{\bf{Example 3.}} By choosing $q=2,4,6,8$, we get the profiles of $u$ with different $q$ at time $t=1,1.2$ and the profiles of $u$ with different $q$ at location $x=1,4$; see Figs. \ref{3a_figure6} and \ref{3a_figure7}. We find that when the advection rate $q$ becomes large, the species density will have a tendency to $0$ for any time and location, which explain the result of Theorem \ref{3a_3atheorem2} with $q\to\infty$.
\begin{figure}[h]
  \centering
  \includegraphics[height=60mm,width=180mm]{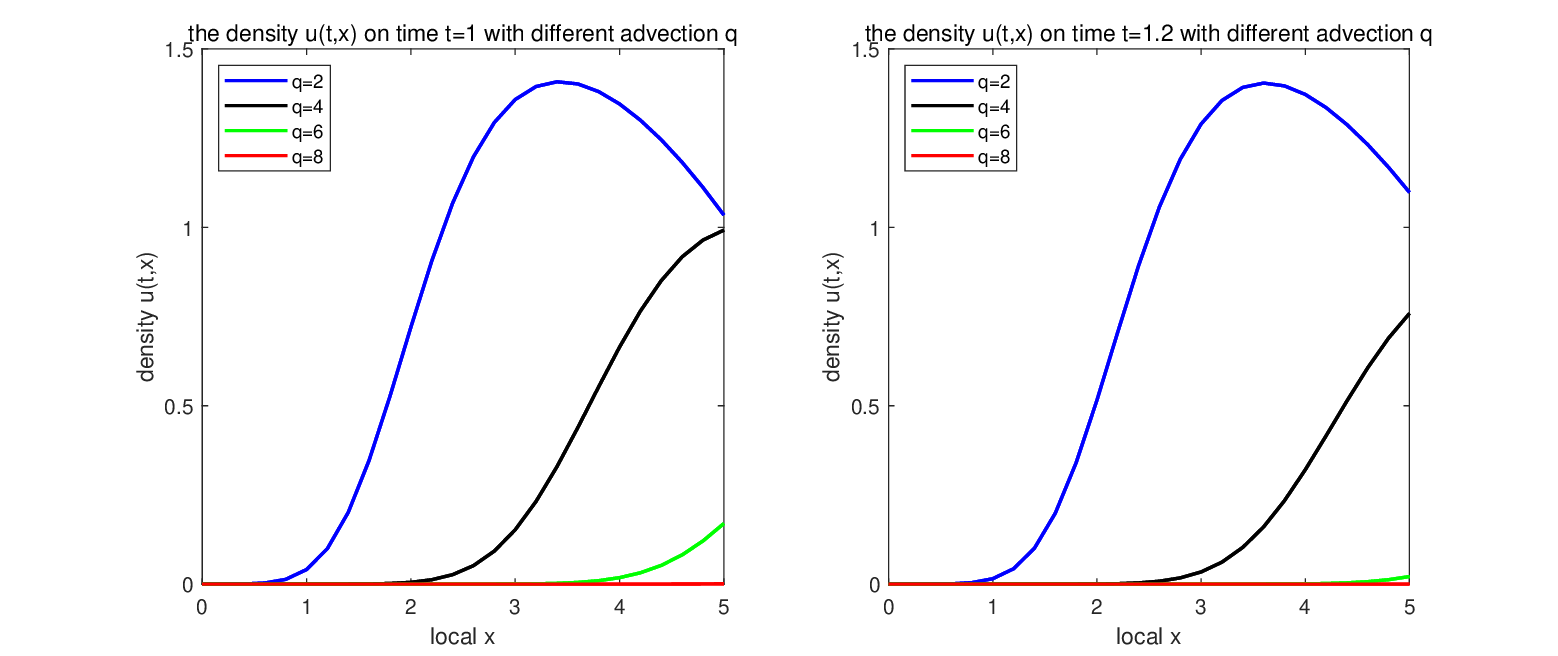}\\
  \caption{The species density of $u$  with different $q$ at time $t=1,1.2$ }\label{3a_figure6}
\end{figure}
\begin{figure}[h]
  \centering
  \includegraphics[height=60mm,width=180mm]{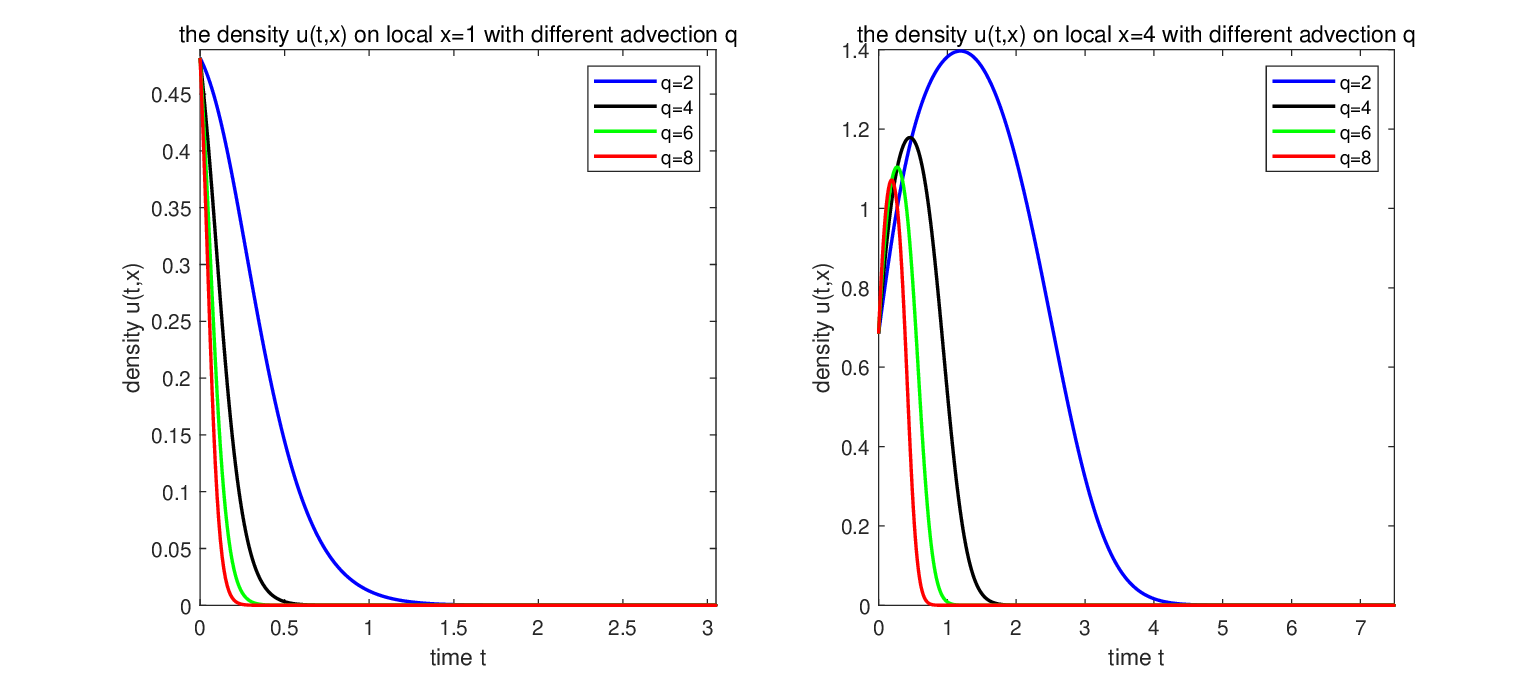}\\
  \caption{The species density of $u$  with different $q$ at location $x=1,4$}\label{3a_figure7}
\end{figure}

\section{Discussion}

The paper introduces a nonlocal reaction-diffusion-advection problem  in a bounded region with  Dirichlet/Neumann boundary condition. Our objective is to investigate its spatial dynamics, namely, the sharp criteria for persistence or extinction and the limiting behaviors of solutions with respect to advection rate. The main conclusions can be summarized as follows:
\begin{enumerate}[(1)]
\item  The existence/nonexistence, uniqueness and stability of nontrivial stationary solutions are obtained. To be specific:
    \begin{enumerate}[(i)]
     \item For Neumann problem \eqref{3a_11}  or  Dirichlet problem  \eqref{3a_1}, if it admits a  bounded nontrivial stationary solution $U(x)$, then $U(x)$ is unique and asymptotically stable.
     \item  \eqref{3a_11} exists a nontrivial stationary solution if and only if $\lambda_p(\mathfrak{L}_{\Omega,J,q}+h(x,0)-d\int_\Omega J(x-y)dy)<0$. Similarly,  \eqref{3a_1} exists a nontrivial stationary solution if and only if $\lambda_p(\mathfrak{L}_{\Omega,J,q}+h(x,0))+d<0$.
      \end{enumerate}

 \item The sharp criteria with respect to advection rate for persistence-extinction are obtained. It can be explained  in biological reality. Namely, we consider fresh water organisms that  subject to
 directional motion in response to water flow with a speed of $q$. There exists a critical threshold value $\tilde{q}$, such that the species can survive only when the velocity of water flow is less than $\tilde{q}$. If the water flow velocity exceeds $\tilde{q}$, the species will become extinct.



\item  The limiting behaviors of solutions with respect to advection rate are considered. It tells us that  a sufficiently large directional motion will cause the species extinction in any situations.
%
\item   The numerical simulations verify our theoretical proof and show that the advection rate has a great impact on the dynamic behaviors of species.

 \end{enumerate}

In addition, when $f(t,x,u)=h(x,u)u$,  we have a remark with sharp criteria for persistence or extinction.  For \eqref{3a_11} and \eqref{3a_1}, notice that the prerequisite of the sharp criteria with respect to $q$  to be true is respectively $h(x,0)\geq d\int_\Omega J(x-y)dy$ and $\inf_{\bar{\Omega}}h(x,0)>d$. So, What should we think about the case for $0\leq h(x,0)< d\int_\Omega J(x-y)dy$ or $0\leq \inf_{\bar{\Omega}}h(x,0)\leq d$? Furthermore, if $h(x,0)$ is sign changed, the sharp criteria with respect to $q$ whether can be established for \eqref{3a_11} and \eqref{3a_1}. It is an open problem and excepts be solved in future. The key to solving this problem is that the monotonicity of the principal eigenvalue of operator $\mathfrak{L}_{\Omega,J,q}+a$  with respect to $q$ for any given $a(x)\in C(\bar{\Omega})\cap L^\infty(\Omega)$. Here, $a(x)$ may be negative.

\section*{Statements and Declarations}
\noindent
{\bf{Competing Interests}} \  We declare that the authors have no competing interests as defined by Springer, or other interests that might be perceived to influence the results and/or discussion reported in this paper.

\section*{Acknowledgement}

 This work is supported in part by the National Natural Science Foundation of China (No. 11871475,\\12271525) and the Fundamental Research Funds for the Central Universities of Central South University (No. CX20230218).
%
%
%
%
%
%
%
%
%
%
%
%

{}

\end{document}